\documentclass[preprint]{elsarticle}
\usepackage{amsmath,amssymb}
\usepackage{hyperref}
\usepackage{graphicx}
\usepackage{lscape}

\newproof{proof}{Proof}
\numberwithin{equation}{section}
\newtheorem{theorem}{Theorem}[section]
\newtheorem{lemma}[theorem]{Lemma}

\newdefinition{definition}[theorem]{Definition}
\newdefinition{example}[theorem]{Example}

\allowdisplaybreaks

\begin{document}

\title
{Time response of a scalar dynamical system with multiple delays via Lambert W functions}

\author[thu]{Shuo-Tsung Chen}
\ead{shough34@yahoo.com.tw}
\author[nchu]{Shun-Pin Hsu}
\ead{shsu@nchu.edu.tw}
\author[thu]{Huang-Nan Huang\corref{cor}}
\ead{nhuang@thu.edu.tw}
\author[nchu]{Bin-Yan Yang}
\ead{youngbinyeah@gmail.com}

\address[thu]{Department of Applied Mathematics, Tunghai University,Taichung 40704, Taiwan}

\address[nchu]{Department of Electrical Engineering, National Chung-Hsing University, Taichung 40704, Taiwan}
\cortext[cor]{Corresponding author}

\begin{abstract}
 In this work, we establish the response of scalar systems with multiple 
 discrete delays based on the Laplace transform. The time response function is expressed as the sum of  
infinite series of exponentials acting on eigenvalues inside
countable branches of the Lambert W functions.
Eigenvalues in each branch of Lambert W function are computed by a numerical iteration.
Numerical examples are presented to illustrate the results obtained.
\end{abstract}


\maketitle

\section{Introduction}

Dynamical systems with delays play an important role
in modeling natural processes which include the influence of past effects
for a better description of the evolution.
In classical physics, life sciences, physiology, engineering system, neural 
networks, epidemiology, and economics, realistic models must take in account 
the time-delays due to the finite propagation speed to determine the future 
evolution. Many examples of real phenomena with delay effects can be found in 
\cite{Kolmanovskii1,Kolmanovskii2,kuang1,kuang2,murray,rws,smith,travis}.
Furthermore of importance in applications, the delay system is, in general,
described by delay differential equations which have
several distinct mathematical properties of ordinary and partial differential 
equations, and also provides them with a purely mathematical interest.

In this work, we consider the differential equation with discrete delays.
Let $N\in\mathbb{N}$, $a, b, h, x_0 \in\mathbb{R}$, not all zero $a_{j d}\in\mathbb{R}$, $1\le j\le N$, and 
$\phi$ is a continuous preshape function defined on $[-N h,0]$.
Consider the following delay differential equation:
\begin{align}
\dot{x}(t) &= a x(t) + \sum_{j=1}^{N} a_{j d} x(t -j h)+b u(t), \notag\\
x(0)&=x_0, \tag{DDE}\label{eqn:MDDE}\\
x(\tau) &= \phi(\tau)~,~-N h \leq \tau < 0 \notag
\end{align}
We allow the discontinuity of the preshape function $\phi$ at $t=0$, i.e., $\phi(0)$ is not necessary equal 
to $x_0$. We want to find analytically the formula to describe the effect of 
input function $u(t)$ to the system's output $x(t)$ instead of just the 
output due to a specified input via numerical method, like \texttt{dde23} in 
Matlab. For those delay systems with different lengths of delay-time, it
can also be transferred into the form given by (\ref{eqn:MDDE}). Thus our 
result can be applied without any difficulty.

For the single delay differential equation, i.e., set $N=1$, $h=1$,  
$a=0$, and $b=0$ into (\ref{eqn:MDDE})):
\[
\dot{x}(t)=a_{1} x(t-1).
\]
Refer to the paper by Coreless \emph{et. al.}, 1996 \cite{Coreless}, it is suggested by using 
Lambert W function to describe the corresponding solution as
\[
x(t)=\sum_{k=-\infty}^{\infty} c_k e^{W_k(a_{1})t}
\]
for any  reasonable  choice of $c_k$ (i.e. such that the sum makes sense). 
Here $W_k$ denotes the $k$-th branch of Lambert W function.

Introduced by Lambert and Euler in 1700's, the Lambert W function
is defined to be any function $W(x)$ satisfies
\begin{equation}
W(x) e^{W(x)} = x
\label{eqn:LambertW}
\end{equation}
Taking the logarithms on both sides of (\ref{eqn:LambertW}) leads to
\[
\log W(x) = \log(x) - W(x)
\]
after some arrangement, and it is obviously that there are infinite many solution due to the property of logarithmic functions in the complex plane. 
When $x\in\mathbb{R}$ and $-1/e\le x<0$, there are two possible real values of $W(x)$ (see Figure \ref{Fig:LambertW}a). Denote the branch satisfying $-1\le W(x)$ by 
$W_0(x)$, or just $W(x)$ when there is no confusion, 
and the branch satisfying $W(x)\le -1$ by $W_{-1}(x)$. The $W(x)$ is referred to as the \emph{principal branch} of the Lambert W function.
When $x$ is complex, we use $W_k(x)$ to denote the $k$-th branch of the Lambert W function whose ranges are shown in Figure \ref{Fig:LambertW}b. 
We can verify that $W_k(\overline{z})=\overline{W_{-k}(z)}$ for all integers
$k$.
Applications of Lambert W function, we refer to Coreless \emph{et. al}, 1996 
\cite{Coreless}, and references there in.
\begin{figure}[h]
	\begin{tabular}{cc}\hspace*{-0.75cm}
		\includegraphics[width=7.25cm]{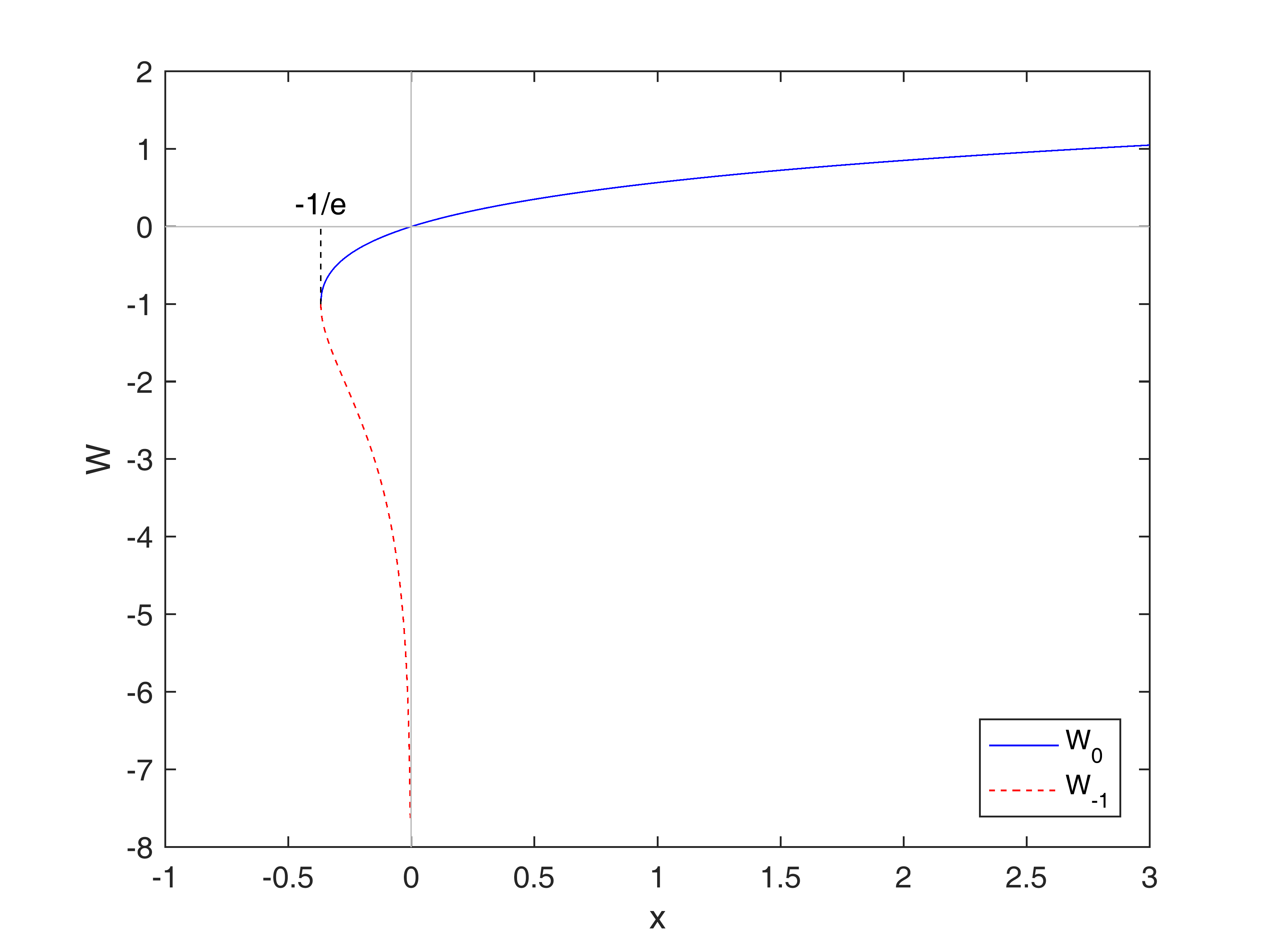} &
		\hspace*{-1.1cm}
		\includegraphics[width=7.25cm]{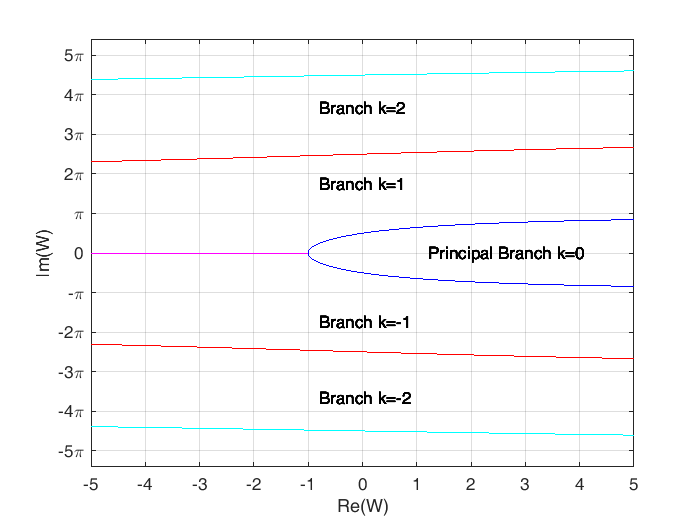}\\
		(a) Two real branches $W_0$ and $W_{-1}$ & (b) Ranges of branches
	\end{tabular}
	\caption{The Lambert W function}
	\label{Fig:LambertW}
\end{figure}

Alternative, consider a single delay system, i.e., $N=1$ in (\ref{eqn:MDDE}) 
and the differential equation becomes
\[
\dot{x}(t) = a x(t) + a_{1d} x(t - h)+b u(t).
\]
Asl and Ulsoy, 2003 \cite{Asl} and Yi \emph{et. al.}, 2010 \cite{SunYi} use Lambert W function to represent
the time response function of this equation and its extension to systems of differential equations, i.e., real numbers $a$ and 
$a_{1d}$ turn to be real matrices $A$ and $A_{1d}$ of size $n\times n$. And
only a single delay is considered in these study.

In this work, we consider for a general $N\in\mathbb{N}$, i.e., multiple 
discrete delays. Firstly, Laplace transform is applied to the 
equation (\ref{eqn:MDDE}) and state transition function is then obtained. The 
associated characteristic equation with the exponential
terms is iteratively solved via Lambert W function with infinite countable branches.
Then the transformed equation is expanded in partial fractional
expansion of the system's eigenvalues, and we can then obtain the analytical expression of time response as the sum of serial combination of exponential functions
acting on different branches of the corresponding Lambert $W$ function.
Numerical examples are provided for illustrative purpose.

This article is organized as follows. In Section 2 we use the Laplace transform to obtain the transformed (\ref{eqn:MDDE}) in
the frequency domain and then find the associated time response
via state transition function.
Section 3 gives a brief introduction of Lambert W functions and then use it to express the state transition function. The eigenvalues of the system is then solved iteratively in terms of Lambert W functions. Numerical
examples are given in section 4 for illustrative purpose. Some
concluding remark is presented in last section.

\section{Time response function by Laplace transform}

\subsection{Laplace transform}
\label{subsec:laplace}

Let $F(s)=\mathcal{L}[f(\cdot);s]=\int_0^\infty f(t) e^{-st} dt$ denote 
the Laplace transform of a exponential order function $f(\cdot)$ defined
on $[-N h,\infty)$ for some $N\in \mathbb{N}$. Let $X(s)$ and $U(s)$ be the 
Laplace transform of the $x(\cdot)$ and $u(\cdot)$, respectively. And
we can also define
\begin{equation}
\Phi_j(s)\triangleq\int_{-j h}^0 \phi(\tau) e^{-s\tau} d\tau
\label{eqn:preshape_Laplace}
\end{equation}
to be the Laplace transform of the preshape function for some $k\in\mathbb{N}$.

Let us extend the preshape function $\phi(\cdot)$ to the whole real line as $\phi_e(\cdot)$:
\[\phi_e(t)=\begin{cases}
0, t>0\\
\phi(t), t\in[-N h,0]\\
0,t<-N h,
\end{cases}
\]
Then its Laplace transform with $1\le j\le N$ is computed as following:
\begin{align*}
\mathcal{L}[\phi_e(\cdot-j h);s]
&=\int_0^\infty \phi_e(t-j h) e^{-st} dt
=\int_{-j h}^\infty \phi_e(\tau) e^{-s\tau} e^{-j s h} d\tau \\
&= e^{-j s h} \int_{-j h}^0 \phi(\tau) e^{-s\tau} d\tau
=e^{-j s h} \Phi_j(s)
\end{align*}
and it follows that 
\begin{align}
\begin{split}
&\int_0^\infty x(t-j h) e^{-s t} dt\\
&=e^{-j s h}\int_{-j h}^0 \phi(\tau) e^{-s \tau} d\tau
+e^{-j s h}\int_{0}^\infty x(\tau) e^{-s\tau} d\tau\\
&=e^{-j s h}\Phi_j(s)+e^{-j s h}X(s)
\end{split}
\label{eqn:2.1}
\end{align}

The corresponding Laplace transform of (\ref{eqn:MDDE}) is given by
\[
s X(s) -x_0 = a X(s) + \sum_{j=1}^{N} \left(a_{j d} e^{-j s h} X(s)+a_{j d} e^{-j s h} \Phi_j(s) \right)+ b U(s)
\]
which implies
\begin{equation}
\left(s-a-\sum_{j=1}^{N} a_{j d} e^{-j s h} \right) X(s) = x_0 + \sum_{j=1}^{N} a_{j d} e^{-j s h} \Phi_j(s) +b U(s)
	\label{eqn:2.2}
\end{equation}
Thus,
\begin{align}
X(s)=\frac{1}{\Delta(s)} \left(x_0+\sum_{j=1}^{N} a_{j d} e^{-j s h}\Phi_j(s)+b U(s)\right)
\stepcounter{equation} \tag{\theequation a} \label{eqn:2.3a}
\end{align}
where
\begin{align}
\Delta(s)=s-a-\sum_{j=1}^{N} a_{j d} e^{-j s h}.
\tag{\theequation b}\label{eqn:2.3b}
\end{align}
where $\Phi_j(s)$ is defined in (\ref{eqn:preshape_Laplace}).
Equation (2.3) gives us the Laplace transform version of the (\ref{eqn:MDDE}).

\subsection{State transition function}
How to find the inverse Laplace transform of (2.3)?
Let
\begin{equation}
\mathcal{L}[\Psi(t);s]\triangleq\frac{1}{\Delta(s)}
\label{eqn:2.4}
\end{equation}
which $\Psi(t)$ denotes the corresponding state transition function of (\ref{eqn:MDDE})
whose definition is described by
\begin{definition}
\label{def:state}
A function $\Psi(t)$ is called the \emph{fundamental solution} or \emph{state transition function}  of (\ref{eqn:MDDE}) if it satisfies the following conditions:	
\begin{align*}
	&\frac{d}{dt}\Psi(t)=a \Psi(t) + \sum_{j=1}^{N} a_{jd} \Psi(t-j h), t>0 \\
	&\Psi(0)=1,
	\Psi(t)=0, t<0.
\end{align*}
\end{definition}

\noindent It is obvious that the state transition function in Definition \ref{def:state} must also
satisfies
\begin{align*}
&\frac{d}{dt}\Psi(t)=\begin{cases}
a \Psi(t), &0< t\le  h \\
a \Psi(t) + a_{1d} \Psi(t-h), &h<t\le 2h \\
a \Psi(t) + a_{1d} \Psi(t-h) + a_{2d} \Psi(t-2 h), &2 h<t\le 3 h \\
\qquad\vdots\\
a \Psi(t) + a_{1d} \Psi(t-h) + \cdots+ a_{N d} \Psi(t-N h), & t> N h \\
\end{cases}\\
&\Psi(0)=1, \quad
\Psi(t)=0, t<0.
\end{align*}
This definition of fundamental solution is the same as Gu \emph{et. al.} 
\cite[check(1.11)]{KeqinGu} and is also much simpler than the one proposed by 
Yi \emph{et. al}, 2010 \cite[check (2.27)]{SunYi} when $N=1$ and Krisztin and Vas, 2011
\cite{Krisztin} for general $N$ but no $a$ terms.

Before we step into finding the inverse Laplace transform, consider the integral given by the lemma:
\begin{lemma} \label{lemma:preshape}
For $1\le j\le N$, the following relationship holds:
\begin{align}
\begin{split}
&	\mathcal{L}\left[\int_{0}^{t} \Psi(t-\tau) a_{j d} \phi_e(\tau-j h) 
	d\tau;s\right] \\
&=	\mathcal{L}\left[\int_{-j h}^{\min\{0,t-j h\}} \Psi(t-\tau-j h) a_{j d} \phi(\tau) d\tau;s\right] \\
&= \frac{e^{-s j h} a_{j d} \Phi(s)}{\Delta(s)}
\end{split}
\end{align}
\end{lemma}
\begin{proof}
Note that by changing of variable, we have 
\begin{align*}
	&\int_{0}^{t} \Psi(t-\tau) a_{j d} \phi_e(\tau-j h) d\tau 
	=\int_{-j h}^{t-j h} \Psi(t-\xi-j h) a_{j d} \phi_e(\xi) d\xi\\
	&=\int_{-j h}^{\min\{0,t-j h\}} \Psi(t-\xi-j h) a_{j d} \phi(\xi) d\xi.
\end{align*}
For a fixed $1\le j\le N$, we let
\begin{align*}
&\mathcal{L}\left[\int_{0}^{t} \Psi(t-\tau) a_{j d} \phi_e(\tau-j h) d\tau;s\right]\\
&=\int_0^\infty \left[\int_{0}^{t} \Psi(t-\tau) a_{j d} \phi_e(\tau-j h) d\tau\right] e^{-s t} dt
\end{align*}
and switching the order of integration leads to
\begin{align*}
&\int_{\tau=0}^\infty \int_{t=\tau}^{\infty} \Psi(t-\tau) a_{j d}\phi_e(\tau-j h) e^{-st}  dt d\tau \\
&=\int_{\tau=0}^\infty \int_{t=\tau}^{\infty} \Psi(t-\tau) e^{-s t}  dt ~~a_{i d}\phi_e(\tau-j h) d\tau \\
&=\int_{\tau=0}^\infty \int_{\xi=0}^{\infty} \Psi(\xi) e^{-s\xi}  d\xi ~~ e^{-s\tau} a_{j d}\phi_e(\tau-j h)d\tau \\
&=\int_{\tau=0}^\infty \frac{1}{\Delta(s)}~ a_{j d} \phi_e(\tau-j h) e^{-s\tau}d\tau\\
&=\frac{a_{j d}e^{-j s h}}{\Delta(s)}  \int_{\xi=-j h}^\infty \phi_e(\xi) e^{-s\xi} d\xi \\
&=\frac{a_{j d}e^{-j s h}}{\Delta(s)}  \int_{\xi=-j h}^{0} \phi(\xi) e^{-s\xi} d\xi\\
&=\frac{a_{j d}e^{-j s h}\Phi_j(s)}{\Delta(s)}  
\end{align*}
\end{proof}

\subsection{Time response function}

Lemma \ref{lemma:preshape}, gives us the time domain correspondence of the 
second term in (\ref{eqn:2.3a}). Thus the time response of (\ref{eqn:MDDE}) 
for $t\ge 0$ is then obtained by finding the time-domain function 
corresponding to equation (\ref{eqn:2.3a}), i.e.,
\begin{align*}
\begin{split}
x(t)&=\Psi(t)x_0
    + \sum_{j=1}^{N} \int_{0}^{t} \Psi(t-\tau) a_{j d} \phi_e(\tau-j h) d\tau
    + \int_{0}^{t} \Psi(t-\tau) b u(\tau) d\tau\\
    &=\Psi(t)x_0
+ \sum_{j=1}^{N} \int_{-j h}^{\min\{0,t-j h\}} \Psi(t-\tau-j h) a_{j d} \phi(\tau) d\tau
+ \int_{0}^{t} \Psi(t-\tau) b u(\tau) d\tau \\
\end{split}
\end{align*}
Based on previous discussion, the solution of delay differential equation can be expressed in terms
of fundamental matrix $\Psi(t)$ is given in the following theorem.
\begin{theorem}\label{thm:2.5.1}
Under appropriate conditions, the solution of the \emph{(\ref{eqn:MDDE})} for $t\ge 0$ is given by the following equation:
\begin{align}
\begin{split}
x(t)&=\Psi(t)x_0
+\sum_{j=1}^N \int_{-j h}^{\min\{0,t-j h\}} \Psi(t-\xi-j h) a_{jd} \phi(\xi) d\xi \\
&\hspace*{0.5cm}+ \int_{0}^{t} \Psi(t-\tau) b u(\tau) d\tau 
\end{split}
\label{eqn:2.5}
\end{align}
\end{theorem}
\proof
For simplicity of the algebraic operation in the proof, we set $N=1$, i.e., 
single delay case, to explore the basic idea. For more general $N$, a similar 
procedure is applicable.
Taking the derivative to (\ref{eqn:2.5}) gives us the following equation:
\begin{align*}
\dot{x}(t)&=\dot{\Psi}(t)x_0 + b u(t) + \int_{0}^{t} \dot{\Psi}(t-\tau) b u(\tau) d\tau  \\
&
\qquad+\begin{cases}
 \int_{-h}^{0} \dot{\Psi}(t-\xi-h) a_{1d} \phi(\xi) d\xi
& \mbox{ if } t>h\\
  a_{1d} \phi(t-h) +
 \int_{-h}^{t-h} \dot{\Psi}(t-\xi-h) a_{1d} \phi(\xi) d\xi
& \mbox{ if } t\le h
\end{cases}
\stepcounter{equation} \tag{\theequation}
\label{eqn:2.6}
\end{align*} 
Since $\xi:-h\to \min\{0,t-h\}$ implies $t-\xi-h:t\to \max\{0,t-h\}$, then
\[
\frac{d}{dt}\Psi(t-\xi-h)=
\begin{cases}
a \Psi(t-\xi-h), 0<t-\xi-h\le h ~(\mbox{or } t-2h\le \xi<\min\{0,t-h\}) \\
a \Psi(t-\xi-h)+a_{1d} \Psi(t-\xi-2h), h<t-\xi-h \le t ~(\mbox{or } -h<\xi\le t-2h) 
\end{cases}
\]
And similarly,
as $\tau:0\to t$ implies $t-\tau:t\to 0$, then
\[
\frac{d}{dt}\Psi(t-\tau)=
\begin{cases}
a \Psi(t-\tau), 0<t-\tau\le h ~(\mbox{or } t-h\le \tau< t) \\
a \Psi(t-\tau)+a_{1d} \Psi(t-\tau-h), h<t-\tau\le t ~(\mbox{or } 0\le \tau < t-h) 
\end{cases}
\]
Thus
\begin{align*}
&\int_{-h}^{\min\{0,t-h\}} \dot{\Psi}(t-\xi-h) a_{1d} \phi(\xi) d\xi \\
&=\begin{cases}
\int_{-h}^{t-2h} \left[a \Psi(t-\xi-h) + a_{1d} \Psi(t-\xi-2h)
\right] a_{1d} \phi(\xi) d\xi \\
\qquad+\int_{t-2h}^{0} a \Psi(t-\xi-h) a_{1d} \phi(\xi) d\xi  
& \mbox{ if } t>h \\
\int_{-h}^{t-h} a \Psi(t-\xi-h) a_{1d} \phi(\xi) d\xi 
& \mbox{ if } t\le h
\end{cases}
\end{align*}
and
\begin{align*}
 \int_{0}^{t} \dot{\Psi}(t-\tau) b u(\tau) d\tau
=\begin{cases}
\int_{0}^{t-h} \left[a \Psi(t-\tau)+a_{1d} \Psi(t-\tau-h)
\right] b u(\tau) d\tau \\
\qquad+\int_{t-h}^{t} a \Psi(t-\tau) b u(\tau) d\tau  & \mbox{ if } t>h \\
\int_{0}^{t} a \Psi(t-\tau) b u(\tau) d\tau & \mbox{ if } t\le h
\end{cases}
\end{align*}
Therefore, the substitution of above derivative terms into (\ref{eqn:2.6}) lead to the following equation:
when $t>h$
 \begin{align*}
 \dot{x}(t)
 &=a \Psi(t) x_0 + a_{1d} \Psi(t-h) x_0 
+\int_{-h}^{t-2h} \left[a \Psi(t-\xi-h) + a_{1d} \Psi(t-\xi-2h)
\right] a_{1d} \phi(\xi) d\xi \\
&+\int_{t-2h}^{0} a \Psi(t-\xi-h) a_{1d} \phi(\xi) d\xi   +\int_{0}^{t-h} \left[a \Psi(t-\tau)+a_{1d} \Psi(t-\tau-h)
 \right] b u(\tau) d\tau\\
& +\int_{t-h}^{t} a \Psi(t-\tau) b u(\tau) d\tau
 + b u(t) \\
 &=a \Psi(t)x_0
 +a \int_{-h}^{t-2h}  \Psi(t-\xi-h) a_{1d} \phi(\xi) d\xi 
 +a \int_{t-2h}^{0}  \Psi(t-\xi-h) a_{1d} \phi(\xi) d\xi  \\
  &+a \int_{0}^{t-h} \Psi(t-\tau) b u(\tau) d\tau
  +a \int_{t-h}^{t} \Psi(t-\tau) b u(\tau) d\tau 
 + a_{1d} \Psi(t-h) x_0 \\
 &\quad+ a_{1d}\int_{-h}^{t-2h} \Psi(t-\xi-2h) a_{1d} \phi(\xi) d\xi
  + a_{1d}\int_{0}^{t-h} \Psi(t-\tau-h) b u(\tau) d\tau
  + b u(t) \\
 &=a \left[\Psi(t)x_0
 +\int_{-h}^{0}  \Psi(t-\xi-h) a_{1d} \phi(\xi) d\xi 
 + \int_{0}^{t} \Psi(t-\tau) b u(\tau) d\tau\right] \\
 &\quad+ a_{1d}\left[ \Psi(t-h) x_0
 + \int_{-h}^{t-2h} \Psi(t-\xi-2h) a_{1d} \phi(\xi) d\xi
  + \int_{0}^{t-h} \Psi(t-\tau-h) b u(\tau) d\tau\right]
  + b u(t) \\
 &= a x(t)+a_{1d} x(t) +b u(t)
 \stepcounter{equation} \tag{\theequation a}
 \label{eqn:2.7a}
 \end{align*} 
 and when $0<t\le h$
 \begin{align*}
 \dot{x}(t)
 &=a \left[\Psi(t)x_0 
  + \int_{-h}^{t-h} \Psi(t-\xi-h) a_{1d} \phi(\xi) d\xi
  + \int_{0}^{t} \Psi(t-\tau) b u(\tau) d\tau
  \right]
  +a_{1d} \phi(t-h) + b u(t)  \\
 &=a x(t) + a_{1d} x(t-h)
 + b u(t)
 \tag{\theequation b}
 \label{eqn:2.7b}
 \end{align*}
 And when $t=0$, $x(0)=\Psi(0)x_0=x_0$.
 This concludes the proof for $i=1$ case. Similarly, we can prove for the case with  $2\le i \le N$. 
\endproof

\section{Approximate Solution via Lambert W Functions}
\subsection{Lambert W function}
Introducesd by Lambert and Euler in 1700's, the Lambert W function
is defined to be any function $W(x)$ satisfies
\begin{equation}
W(x) e^{W(x)} = x
\end{equation}
As sugested by Coreless \emph{et. al.}, 1996 \cite{Coreless} and Als and Ulsoy, 2003 \cite{Asl},
the solution may be expressed in terms of infinite number of 
branches of the Lambert W function. We begin from 
(2.3):
\begin{align}
X(s)
=\frac{1}{\Delta(s)} \left(x_0+\sum_{j=1}^{N} a_{jd} e^{-j s h}\Phi_j(s)+b U(s)\right)
\tag{\ref{eqn:2.3a}}
\end{align}
where
\begin{align}
\Delta(s)=s-a-\sum_{j=1}^{N} a_{jd} e^{-j s h}.
\tag{\ref{eqn:2.3b}}
\end{align}
And the equation
\[
\Delta(s)=0
\]
is called the \emph{characteristic equation} of the system. The zeros of
the characteristic equation are the eigenvalues of the generator of the 
strong continuous semigroup defined by the solution operators \cite{Krisztin}. To find its zeros is corresponding to solving the equation
\begin{align}
 s-a-\sum_{j=1}^{N} a_{jd} e^{-j s h}=0
 \label{eqn:LambertW_1}
\end{align}
which can be re-expressed as
\[
(s-a)h e^{(s-a)h}=\left(\sum_{j=1}^{N}a_{j d} e^{-j s h}\right)h e^{(s-a)h}
=\sum_{j=1}^{N} a_{j d} h e^{-j a h} e^{-(j-1)(s-a)h}.
\]
Suppose $s_{j,k}$ is the root of the characteristic equation, i.e.,
$\Delta(s_{j,k})=0$, $s_{j,k}$ must also satisfy the Lambert W equation
\begin{align}
s_{j,k} = \frac{1}{h} W_k\left(\sum_{j=1}^{N} a_{j d} h e^{-j a h} e^{-(j-1)(s_{j,k} -a)h}\right)+a,\quad 1\le j\le N
\label{eqn:LambertW_Sk}
\end{align}
where $W_k(z)$ denotes the $k$-th branches of Lambert W function, which is a one-one function for $z\not=0\in\mathbb{C}$ (Note $W_0(0)=0$ is in the principal branch), thus $s_{j,k}\not=s_{i,\ell}$ when $k\not=\ell\in\mathbb{Z}$ for any $1\le i,j\le N$. It is obviously true that the conjugate of the solution of $\Delta (s)=0$ is also an another solution whose proof is given below.
\begin{lemma}
	Suppose 
$\Delta(S)=0$ for some complex number, then $\Delta(\bar{S})=0$ also holds where
	$\bar{S}$ is the complex conjugate of $S$.
\end{lemma}
\begin{proof} Let $S=\alpha+i \beta$ be a solution of (\ref{eqn:LambertW_1}) where $\alpha, \beta\in\mathbb{R}$ and be substituted into (\ref{eqn:LambertW_1}):
\begin{align*}
	0 & = (\alpha+i\beta)-a-\sum_{j=1}^{N} a_{jd} 
	e^{-j (\alpha+i \beta) h} \\
	& = \alpha +i \beta-a-\sum_{j=1}^{N} a_{j d}  
	e^{-j \alpha h} \left( \cos j \beta h -i \sin j \beta h\right) \\
	& = \alpha -a - \sum_{j=1}^{N} a_{j d}	e^{-j \alpha h}  \cos j \beta h 
	+i \left( \beta+\sum_{j=1}^{N} a_{j d}  
	e^{-j \alpha h} \sin j \beta h\right) \\
\end{align*}
thus
\begin{align}
	\alpha = a + \sum_{j=1}^{N} a_{j d}	e^{-j \alpha h}  \cos j \beta h,
	\quad
	\beta = -\sum_{j=1}^{N} a_{j d} e^{-j \alpha h} \sin j \beta h.
	\label{eqn:3.4}
\end{align}
And taking the conjugate of right hand side of (\ref{eqn:LambertW_1}) gives us
\begin{align*}
 & \overline{(\alpha+i\beta)}-a-\sum_{j=1}^{N} a_{jd} \ 
      \overline{e^{-j (\alpha+i \beta) h}} \\
	& = \alpha -i \beta-a-\sum_{j=1}^{N} a_{j d}  
	e^{-j \alpha h} \left( \cos j \beta h +i \sin j \beta h\right) \\
	& = \alpha -a - \sum_{j=1}^{N} a_{j d}	e^{-j \alpha h}  \cos j \beta h 
	-i \left( \beta+\sum_{j=1}^{N} a_{j d}  
	e^{-j \alpha h} \sin j \beta h\right) =0
\end{align*}
by using the relationship (\ref{eqn:3.4}). Hence $\bar{S}$ is also a solution.
\end{proof}

\noindent\textbf{Remark 1:} From this lemma, when $s_{j,k}$, $k>0$ is a solution, then $s_{j,-k}=\bar{s}_{j,k}$ will be also a solution. When $k=0$ and $s_{j,0}$ is a real number, then there are only one eigenvalue, i.e., $s_{1,0}=s_{2,0}=\cdots=s_{N,0}$, corresponding to this principle branch $W_0$. Otherwise when $s_{j,0}$ is a complex number, then $s_{1,0}, s_{2,0}
\ldots, s_{N,0}$ will becomes one complex number plus its complex conjugate.  

\noindent\textbf{Remark 2:} For simplicity, we can renumber the roots $s_{j,k}$, $1\le j\le N$ and $k\in\mathbb{Z}$, of $\Delta(s)=0$ as $S_{n}$, with $n\in\mathbb{Z}$ and $S_0=s_{1,0}$ being the number with largest real part. If $S_{0}$ is a real number, then $S_{-n}=\bar{S}_{n}$ for $n>0$ and otherwise $S_{-(n+1)}=\bar{S}_{n}$ for $n\ge 0$. Thus $S_{n}$ with $n>0$ are located inside the upper half complex plane.
 
\noindent\textbf{Remark 3:} Once $S_{0}$ is determined, then the stability of the linear delay-equation (\ref{eqn:MDDE}) can be easy determined by check its real part to be positive or not (i.e., when $S_0$ has negative real part, then the corresponding system is stable). 

We have use the following property in the discussion of \textbf{Remarks 2} and \textbf{3}, 
\begin{lemma}\cite[Lemma 3]{Shinozaki}
For arbitrary $z\in\mathbb{C}$,
\[
\max\{\mathrm{Re}(W_k(z))|k=0,\pm1,\pm 2\ldots\}=\mathrm{Re}(W_0(z))
\]
is satisfied.
\end{lemma}

\subsection{Time response function}
Based on previous discussion, we then can write
\[
\frac{1}{\Delta(s)} 
= \sum_{n=-\infty}^\infty \frac{C_n}{s-S_n}
\]
and the coefficient $C_k$ is then determined by the following Lemma:

\begin{lemma}
\label{lem:expansion}
Suppose $a$, $a_{j d}$ with $1\le j\le N$, and $h\in\mathbb{R}$ then for any $s\in\mathbb{C}$
\begin{align*}
\frac{1}{\Delta(s)} = \frac{1}{ s-a-\sum_{j=1}^{N} a_{j d} e^{-j s h}} 
= \sum_{n=-\infty}^\infty \frac{C_n}{s-S_n}
\stepcounter{equation}\tag{\theequation a}
\end{align*}
where
\begin{align*}
C_n=\frac{1}{1+\sum_{j=1}^{N} j a_{j d} h e^{-j S_n h}},\quad
S_n = \frac{1}{h} W_k\left(\sum_{j=1}^{N} a_{j d} h e^{-j a h} h e^{-(j-1)(S_{n}-a)h}\right)+a
\tag{\theequation b}
\label{eqn:LambertW_S_k_1}
\end{align*}
where $W_k$ is the branches where $S_n$ belongs to.
\end{lemma}
\begin{proof}
Suppose $S_n$ is given by (\ref{eqn:LambertW_S_k_1}), then we have
\[
\Delta(S_n)= S_n-a-\sum_{j=1}^N a_{j d} e^{-j S_n h}=0,
\]
i.e., $S_n$ is the solution of the characteristic equation
for all $n\in\mathbb{Z}$. Let $s=S_\ell$ with any $\ell\in \mathbb{Z}$ and consider
\begin{align*}
&\sum_{k=-\infty}^\infty \left(S_\ell-a-\sum_{j=1}^N a_{j d} e^{-j S_\ell h}\right)\frac{C_k}{S_\ell-S_k} \\
& =\sum_{k=-\infty}^\infty \frac{S_\ell-a-\sum_{j=1}^N a_{j d} e^{-j S_\ell  h}}{S_\ell-S_k}\frac{1}{1+\sum_{j=1}^{N} j a_{j d} h e^{-j S_k h}}\\
&=\left(\lim_{s\to S_\ell} \frac{s-a-\sum_{j=1}^N a_{j d} e^{-j s h}}{s-S_\ell}\right)\frac{1}{1+\sum_{j=1}^{N} j a_{j d} h e^{-j S_\ell h}} \\
&=1.
\end{align*}
Next consider $s\not=S_k$ for all $k\in\mathbb{Z}$, and pick one $\ell \in\mathbb{Z}$ then $S_\ell$ is a simple pole of $\overline{\Psi}(s)$ and
the associated Laurent's expansion of $\overline{\Psi}(s)$ about $s=S_\ell$
is given by
\begin{align}
\overline{\Psi}(s) = \frac{a_{-1}}{s-S_\ell} +a_0 +a_1 (s-S_\ell)
+a_2 (s-S_\ell)^2+\cdots
\label{eqn:Laurent1}
\end{align}
where $a_{-1}$ is the residue of $\overline{\Psi}(s)$, i.e.,
\[
a_{-1}=\lim_{s\to S_\ell} (s-S_\ell)\overline{\Psi}(s)=C_\ell.
\]
Since for $k\not=\ell$,
\begin{align*}
\frac{1}{S_k-s} &= \frac{1}{S_k-S_\ell-(s-S_\ell)} \\
                &= \frac{1}{S_k-S_\ell}\frac{1}{1-\frac{s-S_\ell}{S_k-S_\ell}}
                =\frac{1}{S_k-S_\ell}\left[ 1+\frac{s-S_\ell}{S_k-S_\ell}+
                \frac{(s-S_\ell)^2}{(S_k-S_\ell)^2}+\cdots\right]\\
       &=\sum_{m=0}^{\infty} \frac{1}{(S_k-S_\ell)^{m+1}}(s-S_\ell)^m
\end{align*}
Then
\begin{align*}
&\sum_{k=-\infty}^\infty \frac{C_k}{s-S_k}
=\frac{C_\ell}{s-S_\ell}+\sum_{\stackrel{k=-\infty}{k\not=\ell}}^\infty C_k\frac{1}{s-S_k}\\
&=\frac{C_\ell}{s-S_\ell}+\sum_{\stackrel{k=-\infty}{k\not=\ell}}^\infty (-C_k)\frac{1}{S_k-s} \\
&=\frac{C_\ell}{s-S_\ell}+\sum_{\stackrel{k=-\infty}{k\not=\ell}}^\infty (-C_k)\sum_{m=0}^{\infty} \frac{1}{(S_k-S_\ell)^{m+1}}(s-S_\ell)^m \\
&=\frac{C_\ell}{s-S_\ell}+\sum_{m=0}^{\infty}
\left[\sum_{\stackrel{k=-\infty}{k\not=\ell}}^\infty \frac{-C_k}{(S_k-S_\ell)^{m+1}}\right](s-S_\ell)^m
\end{align*}
Thus the series
\begin{align}
\frac{C_\ell}{s-S_\ell}+\sum_{m=0}^{\infty}
\left[\sum_{\stackrel{k=-\infty}{k\not=\ell}}^\infty \frac{-C_k}{(S_k-S_\ell)^{m+1}}\right](s-S_\ell)^m
\label{eqn:Laurent2}
\end{align}
is also a Laurent series of $\overline{\Psi}(s)$ about 
$s=S_\ell$. By the uniqueness of Laurent's series of a function 
about its simple plot, we must have (\ref{eqn:Laurent1}) is the 
same as (\ref{eqn:Laurent2}), i.e., for any $s$ not equal any 
$S_k$, $k\in\mathbb{Z}$,
\[
\frac{1}{ s-a-\sum_{j=1}^{N} a_{j d} e^{-j s h}} 
= \sum_{k=-\infty}^\infty \frac{C_k}{s-S_k}
\]
where
\begin{align*}
C_k&=\lim_{s\to S_k} (s-S_k) \Psi(s) \\
&=\lim_{s\to S_k} \frac{(s-S_k)}{s-a-\sum_{j=1}^{N} a_{j d} e^{-j s h}} \\
&=\frac{1}{1+\sum_{j=1}^{N} j a_{j d} h e^{-j S_k h}}
\end{align*}
This concludes our proof. 
\end{proof}

By Lemma \ref{lem:expansion}, the state transition function $\Psi(t)$ can be expressed as
\begin{equation}
\Psi(t) = \sum_{k=-\infty}^{\infty} C_k e^{S_k t}
\end{equation}
where $C_k$ and $S_k$ are given in (\ref{eqn:LambertW_S_k_1}). This equation
is of the form firstly proposed by Bellman and Cooke, 1963 \cite{Bellman}. 
Similarly, the second terms in (\ref{eqn:2.3a}) can also be expressed by
\[
\overline{\Psi}(s) \sum_{j=1}^{N} a_{j d} e^{-j s h} \Phi_j(s)
= \sum_{k=-\infty}^\infty \frac{C^I_k}{s-S_k}
\]
and the coefficient $C^I_k$ is then determined by
\[
C^I_k = \lim_{s\to S_k} (s-S_k)\Psi(s) \sum_{j=1}^{N} a_{j d} e^{-j s h} \Phi_j(s)
\]
i.e.,
\begin{equation}
C^I_k=\frac{\sum_{j=1}^{N} a_{j d} e^{-j S_k h}\Phi_j(S_k) }{1+\sum_{j=1}^{N} j a_{j d} h e^{-j S_k h}}
\end{equation}
Therefore the corresponding Laplace transform of (\ref{eqn:MDDE}) becomes
\begin{align*}
X(s)=
 \sum_{k=-\infty}^\infty \frac{C_k}{s-S_k} x_0
+\sum_{k=-\infty}^\infty  \frac{C^I_k}{s-S_k}
+\mathcal{L}[\Psi(t)* b u(t);s]
\end{align*}
and its time domain response is then given by
\begin{equation}
x(t)=
\sum_{k=-\infty}^\infty C_k e^{S_k t} x_0
+\sum_{k=-\infty}^\infty  C^I_k e^{S_k t}
+\int_0^t \sum_{k=-\infty}^\infty C_k e^{S_k (t-\tau)} b u(\tau) d\tau
\label{eqn:time_response}
\end{equation}
which is the characterization of (\ref{eqn:2.5}) in Theorem \ref{thm:2.5.1}.
It is noted that when $N=1$, (\ref{eqn:time_response}) is equal to the formula given by Yi \emph{et. al.} \cite[(2.37)]{SunYi}.

Since there are infinite-countable branches of Lambert W function, we can approximate the time response with $2N+1$ (or $2N$) branches of the Lambert W function for application, i.e.,
\begin{align}
x(t)\approx
\sum_{k=-M}^M C_k e^{S_k t} x_0
+\sum_{k=-M}^M  C^{I}_k e^{S_k t}
+\int_0^t \sum_{k=-M}^M C_k e^{S_k (t-\tau)} b u(\tau) d\tau
\end{align}
where $M\in \mathbb{N}$.
We note that the upper limit of summation may be equal to $M-1$ depending on the value of $S_0$. If $S_0$ is real, then the upper limit is $M$ instead of $M-1$ and in this case when $S_{-1}=\overline{S}_0$.

\subsection{Computation of $S_n$} 
The only left problem is how to compute the value of $s_{j,k}$ (or equivalently, $S_n$) for 
$k\in\mathbb{Z}$ to satisfy (\ref{eqn:LambertW_Sk}). 
When $N=1$, i.e., for a single delay DDE,
\begin{equation}
S_k=s_{1,k}= \frac{1}{h} W_k(a_{1 d} h e^{-a h})+a
	\label{eqn:LambertW_Sk_N=1}
\end{equation}
i.e., $S_k$ is equal to the $k$-th branch of the Lambert W function
acting on the real number $a_{1 d} h e^{-a h}$.

When $N>1$,we use the \texttt{fsolve} command in Matlab to compute $s_{j,k}$
as the solution of the following equation expressed by Lambert
W functions: 
\begin{equation}
(s-a)h- W_k\left(\sum_{j=1}^{N} a_{j d} h e^{-j a h} e^{- (j-1) (s-a) h} \right) = 0. 
\label{eqn:fsolve}
\end{equation}
with initial guess based on (\ref{eqn:LambertW_Sk_N=1}), for example we
can choose:
\begin{equation}
s_{j,k}^{(0)}= \frac{1}{h} W_k\left(a_{1 d} h e^{-a h}\right)+a + i (j-2)
\begin{cases} \pi, & k=0 ,\\ 2\pi, & k> 0,\end{cases}
\end{equation}
when $N=3$, i.e., by adjusting the imaginary part of $s_{j,k}^{(0)}$ 
such that the \texttt{fsolve} command in Matlab converges. Since the roots are complex pair if it is a complex number, we only need to 
perform the computation for $k\ge 0$. The following two examples demonstrate the computation of $S_n$ for two and three delays cases, respectively.

\begin{example} \label{exmp1}
	For the two delays ($N=2$) system with unit delay $h=1$, we consider the following two cases:

\begin{description}
	\item[Case (a)] Given parameters $a=-1$, $a_{1d}=-1$, $a_{2d}=-1/2$. The characteristic equation in terms of the $k$-th branch
of Lambert W function is given by:
\[
        s+1-W_k\left( -e-\frac12  e^{1-s}\right)=0
\]
with initial guess
\begin{equation*}
	s_{j,k}^{(0)} =W_k(-e)-1 - i (j-1)
	\begin{cases} \pi, & k=0, \\2 \pi,  & k>0.\end{cases}
\end{equation*}
For $k=0$ the initial guess is $s_{j,0}^{(0)}=-0.60502+1.78819-(j-1)\pi$ and Matlab's \texttt{fsolve} command gives us
\[
S_0=s_{1,0}\approx-0.27495+ 1.47517 i,\quad
S_{-1}=s_{2,0}\approx-0.27495- 1.47517 i
\]
by taking $4$ and $3$ iterations
to achieve the iteration error within $10^{-10}$ and $S_0$ is a complex number.
Similarly, the central 10 values of $S_n$'s are computed and listed below:
\begin{align*}
	S_{0}=s_{1,0},S_{-1}=s_{2,0} & \approx -0.27495\pm1.47517i,\\
	S_{1}=s_{1,1},S_{-2}=s_{1,-1}& \approx -1.14682\pm7.24009i,\\	
	S_{2}=s_{2,1},S_{-3}=s_{2,-1}&\approx -1.27049\pm3.64513i,\\	
	S_{3}=s_{1,2},S_{-4}=s_{1,-2}&\approx -1.50747\pm13.4657i,  \\
	S_{4}=s_{2,2},S_{-5}=s_{2,-2}&\approx -1.66428\pm10.0261i.
\end{align*}
And the distribution of roots of the characteristic equation for branches $n=-5,\ldots,5$ are shown in Fig. \ref{Fig_LambertW_2_dist}(a). 

\item[Case (b)] Given parameters become $a=-1$, $a_{1d}=1/2$, $a_{2d}=1/4$. The characteristic equation is then given by: 
\[
	 s+1-W_k\left( \frac12 e+\frac14  e^{1-s}\right)=0
\]
with same initial guess:
\[
	s_{j,0}^{(0)} =W_k(\frac{1}{2}e)-1 - i (j-1) 
	\begin{cases}  \pi,  & k=0, \\ 2 \pi, & k>0.\end{cases}
\]
The central 11 values of $S_n$'s are computed by using \texttt{fsolve} with at most $7$ iterations to achieve the iteration error within $10^{-10}$ and are listed below:
	\begin{align*}
		S_{0}=s_{1,0} & \approx -0.11929,\\
		S_{1}=s_{2,1},S_{-1}=s_{2,-1}&\approx -1.36927\pm2.51759i,\\		S_{2}=s_{1,1},S_{-2}=s_{1,-1}& \approx -1.37965\pm5.30446i,\\		
		S_{3}=s_{1,2},S_{-3}=s_{1,-2}&\approx -1.82137\pm11.63900i,  \\
		S_{4}=s_{2,2},S_{-4}=s_{2,-2}&\approx -1.89204\pm8.71328i,  \\
		S_{5}=s_{1,3},S_{-5}=s_{1,-3}&\approx -2.05668\pm17.94914i, \\
		S_{6}=s_{2,3},S_{-6}=s_{1,-3}&\approx -2.13589\pm14.97910i.
	\end{align*}
	And the distribution of roots for branches $k=-5,\ldots,5$ are shown in Fig. \ref{Fig_LambertW_2_dist}(b).
\end{description}
\begin{figure}[h]
	\begin{tabular}{cc}\hspace*{-0.5cm}
		\includegraphics[width=8.2cm]{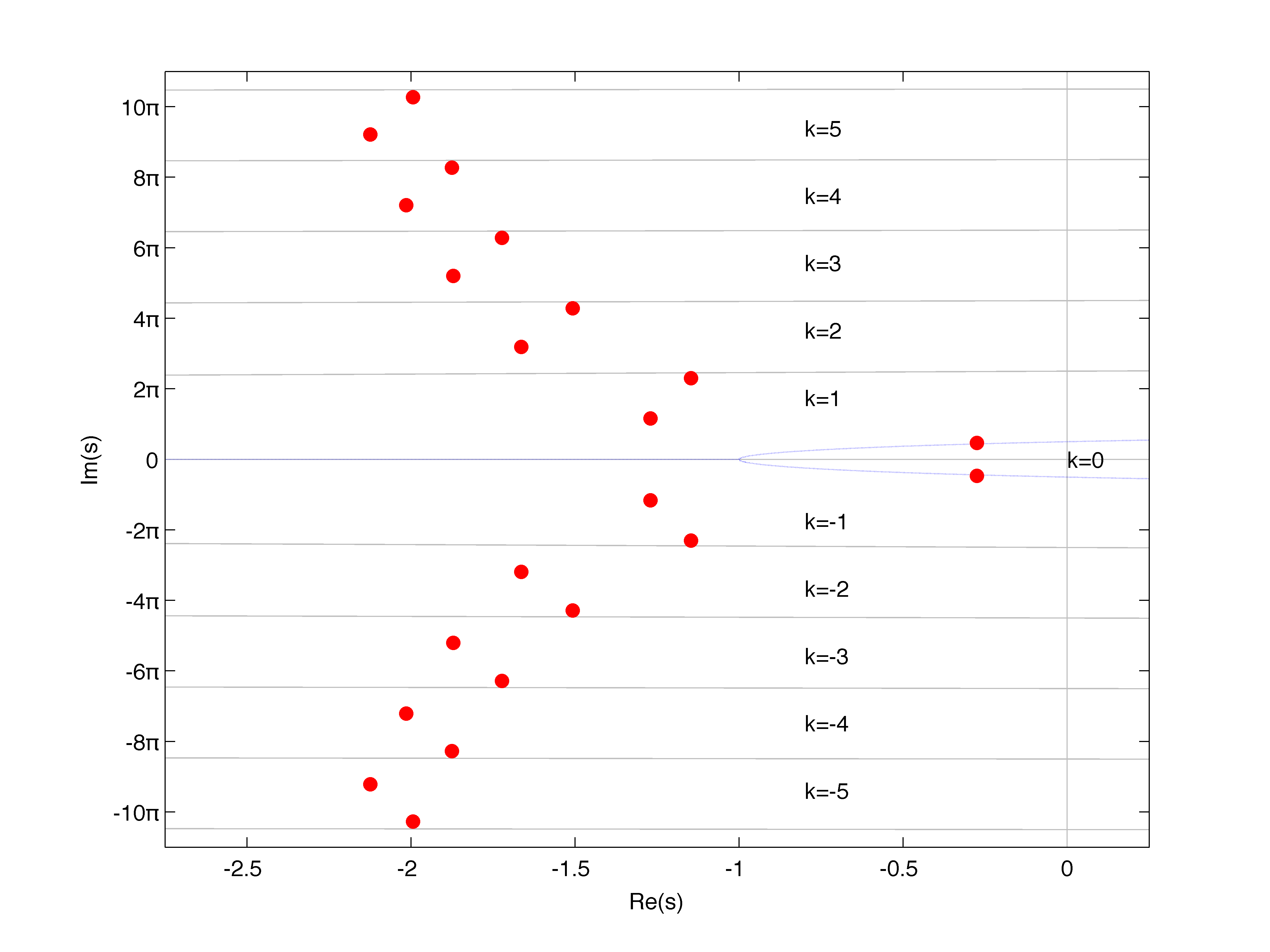} &\hspace*{-1.2cm}
		\includegraphics[width=8.2cm]{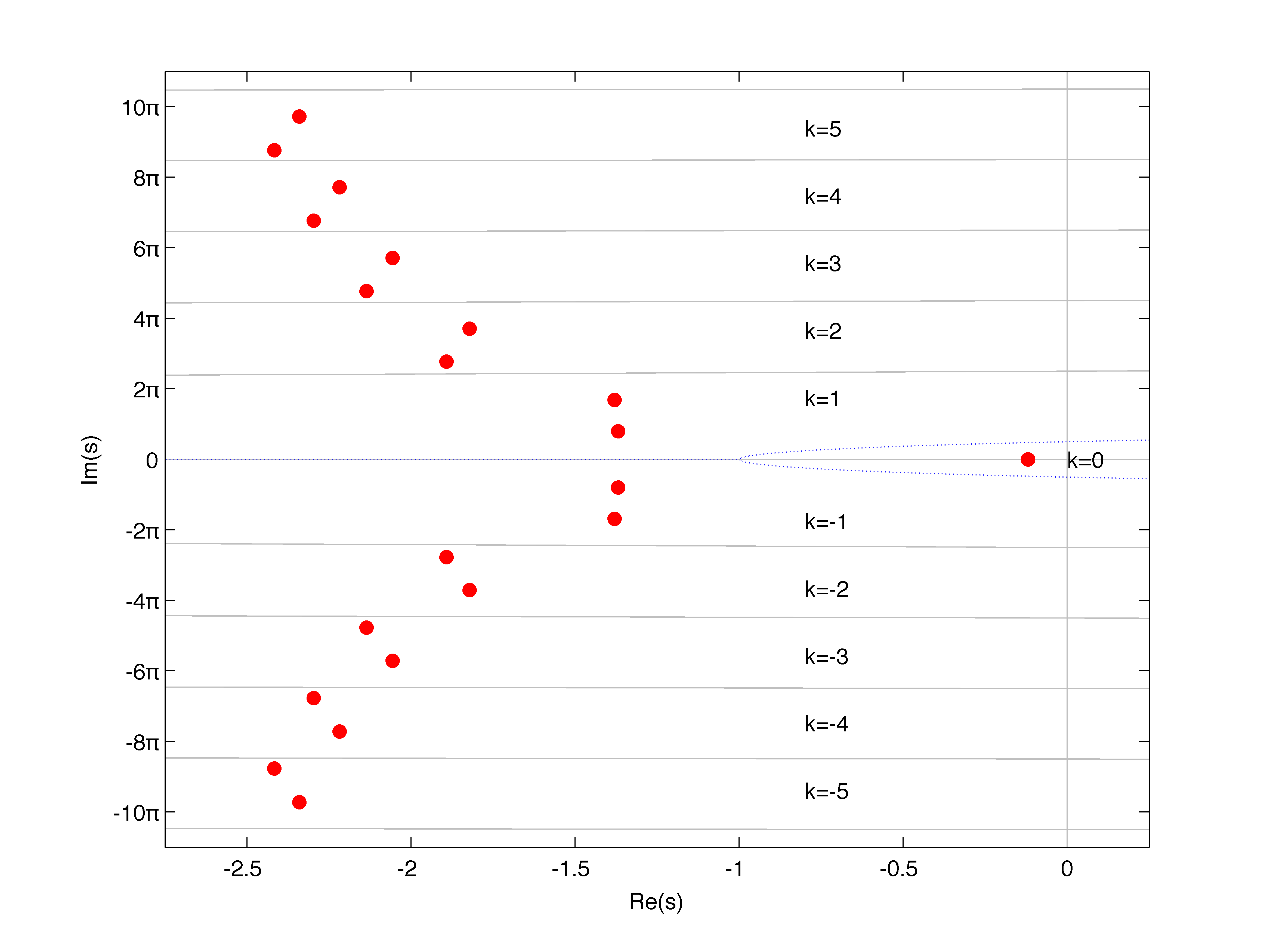}\\
		(a) $a=-1$, $a_{1d}=-1$, and $a_{2d}=-1/2$ & (b) $a=-1$, $a_{1d}=1/2$, and $a_{2d}=1/4$
	\end{tabular}
	\caption{Roots of the characteristic equation for $k=-5,\ldots,5$.}
	\label{Fig_LambertW_2_dist}
\end{figure}
As shown in this Fig. \ref{Fig_LambertW_2_dist}, it is evident that there are two roots in each 
branches of Lambert W function and roots behave in the manner with almost constant slopes of change. Since the value $S_0$ of both systems has negative real part, these systems are stable.

\end{example}
\begin{example} \label{exmp2}
For the three delays ($N=3$) equation
with given parameters $a=-1$, $a_{1d}=1/2$, $a_{2d}=-1$, $a_{3d}=-1$, and the delay
time $h=1$, the equation for calculate the characteristic equation in terms of the $k$-th branch
of Lambert W function is given by:
\[
        s+1-W_k\left( \frac12 e-e^{1-s}-e^{1-2s}\right)=0
\]
with initial guess
\begin{equation*}
	s_{j,k}^{(0)} =W_k(\frac12e)-1 + i (j-2)
	   \begin{cases}
	         \pi/2, & k=0, \\
	         3\pi,  & k>0.
	   \end{cases}
\end{equation*}
The central 14 values of $S_n$'s are computed by using \texttt{fsolve} with at most $11$ iterations to achieve the iteration error within $10^{-10}$ which are listed below:
	\begin{align*}
		S_{0}=s_{2,0},S_{-1}=s_{1,0} & \approx  0.08945\pm0.86785i,\\
		S_{1}=s_{2,1},S_{-2}=s_{2,-1}& \approx -0.41586\pm4.56389i,\\
		S_{2}=s_{3,1},S_{-3}=s_{3,-1}& \approx -0.52264\pm6.86847i,  \\
	    S_{3}=s_{1,1},S_{-4}=s_{1,-1}& \approx -0.62749\pm2.84865i,\\		
		S_{4}=s_{2,2},S_{-5}=s_{2,-2}& \approx -0.74691\pm10.85396i,  \\
		S_{5}=s_{3,2},S_{-6}=s_{3,-2}& \approx -0.75648\pm13.13363i, \\
		S_{6}=s_{1,2},S_{-5}=s_{1,-2}& \approx -0.89966\pm9.02658i.
	\end{align*}
	And the distribution of roots for branches $k=-5,\ldots,5$ are shown in Fig. \ref{Fig_LambertW_3_dist}.
\begin{figure}[h]
		\includegraphics[width=8.2cm]{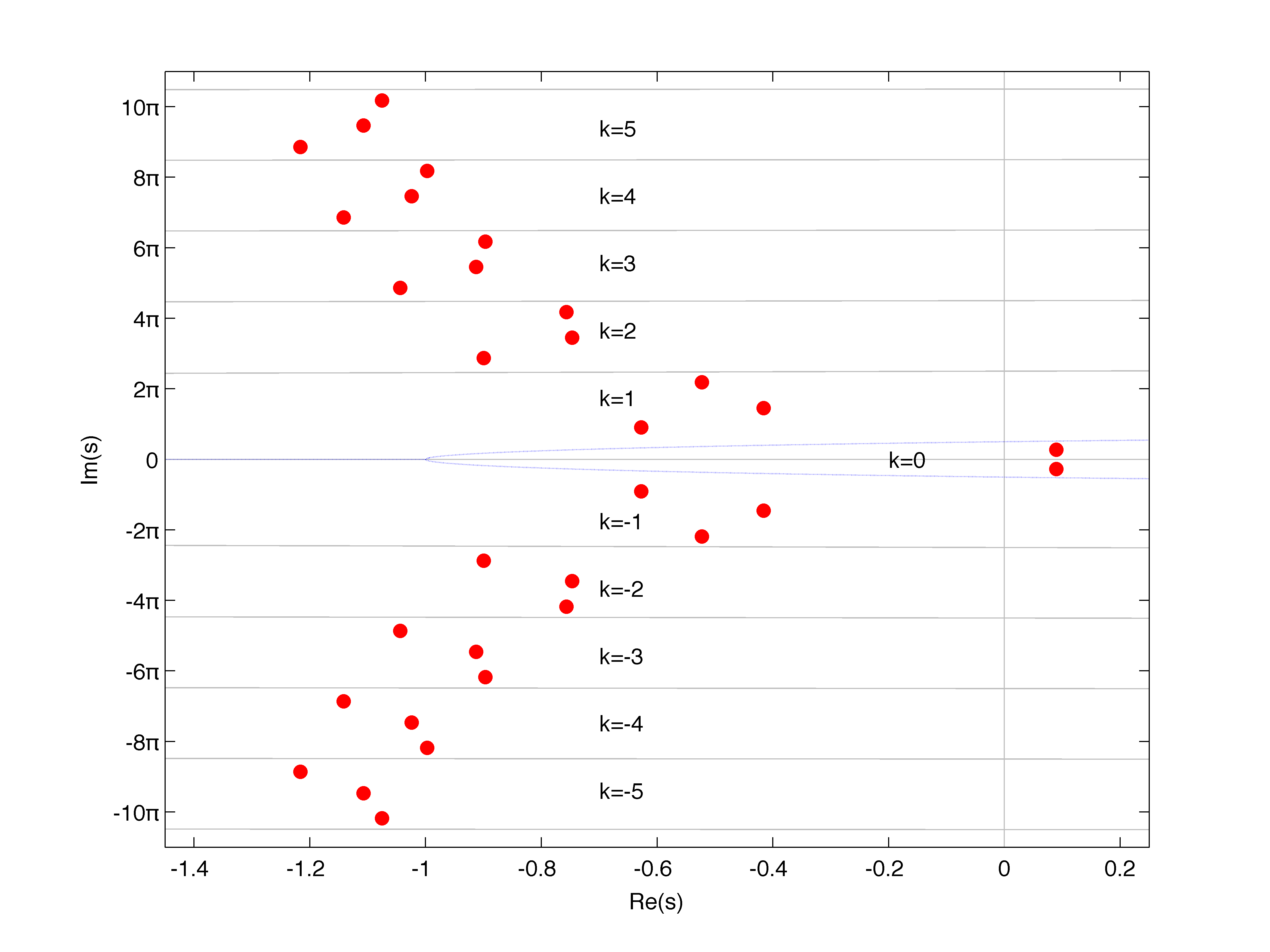} 
	\caption{Roots of the characteristic equation for $k=-5,\ldots,5$ when  $a=-1$, $a_{1d}=1/2$, $a_{2d}=-1$, $a_{3d}=-1$.}
	\label{Fig_LambertW_3_dist}
\end{figure}
As shown in this Fig. \ref{Fig_LambertW_3_dist}, it is evident that there are three roots in each 
branches of Lambert W function except for $k=0$ branch which has two roots inside. And since $\mathrm{Re}(S_0)>0$, this system is unstable.
\end{example}

\section{Illustrative Examples}
In this section, examples with single, two and three delays cases are presented to illustrate the application
of our method to compute their corresponding time responses. Since there are
some references deal with the time response of single delay system, we just
present the basic information to construct the approximate solution. But for
two-delay system, more detail discussion will be addressed.

\subsection{Single delay case}
Consider the following DDE:
\begin{align*}
\dot{x}(t)=-x(t)- x(t-1)+ \cos t, t>0 \\
x(0)=x_0=1,\quad x(t)=\phi(t)=1, t\in[-1,0)
\end{align*}
The corresponding eigenvalue $S_k$ is computed by the Lambert W function
$
S_k = W_k(-e)-1
$
and $C_k=1/(1-e^{-S_k})$ for $k=-3,-2,\ldots,2$.
The state transition function is then given by
\[
\Psi(t)=\sum_{k=-\infty}^{\infty} C_k e^{-S_k t}
\approx\sum_{k=-3}^{k=2} C_k e^{-S_k t}
\]
Also, since
\[
\Phi_1(s)=\int_{-1}^0 \phi(t) e^{-st} dt=\frac{e^s-1}{s}
\]
and
\[
C^I_k=(1-C_k) \frac{\Phi_1(S_k)}{h}
\]
and thus the total response of this DDE is then given by
\[
x(t)\approx \sum_{k=-3}^{2} (C_k x_0+C^I_k) e^{-S_k t}
       +\sum_{k=-3}^{2} C_k \int_0^t e^{-S_k (t-\tau)} \cos(\tau) d\tau
\]
and the first part of $x(t)$ is called the initial condition response.
The corresponding values for $S_k$ and $C_k$ are listed below:
\begin{equation*}
\begin{array}{lll}
S_{0,-1}\approx -0.6050\pm1.7882i, \quad
&S_{1,-2}\approx -2.0528\pm7.7184i,  \quad
&S_{2,-3}\approx -2.6474\pm14.0202i; \\
C_{0,-1}\approx 0.2712\mp0.3477i,  \quad
&C_{1,-2}\approx -0.0008867\mp0.1296i,  \quad
&C_{2,-3}\approx -0.003286\mp0.07117i.
\end{array}
\end{equation*}
The initial condition and total responses of this single delay system are given in Fig. \ref{Fig_LambertW_1}.

\begin{figure}[htb]
	\hspace*{-2em}
	\begin{tabular}{cc}
	\includegraphics[width=7.25cm]{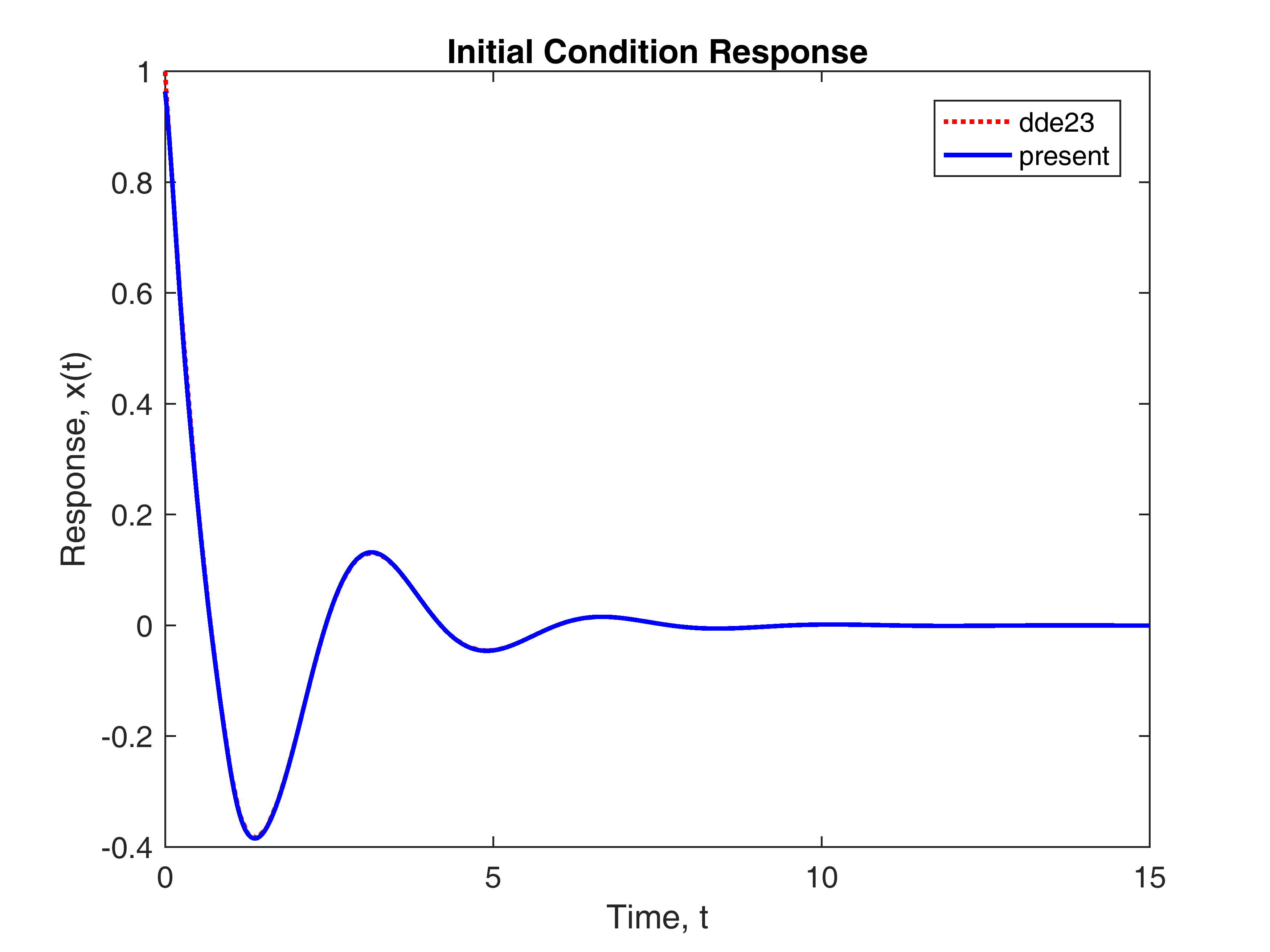} &
	\hspace*{-3em}
	\includegraphics[width=7.25cm]{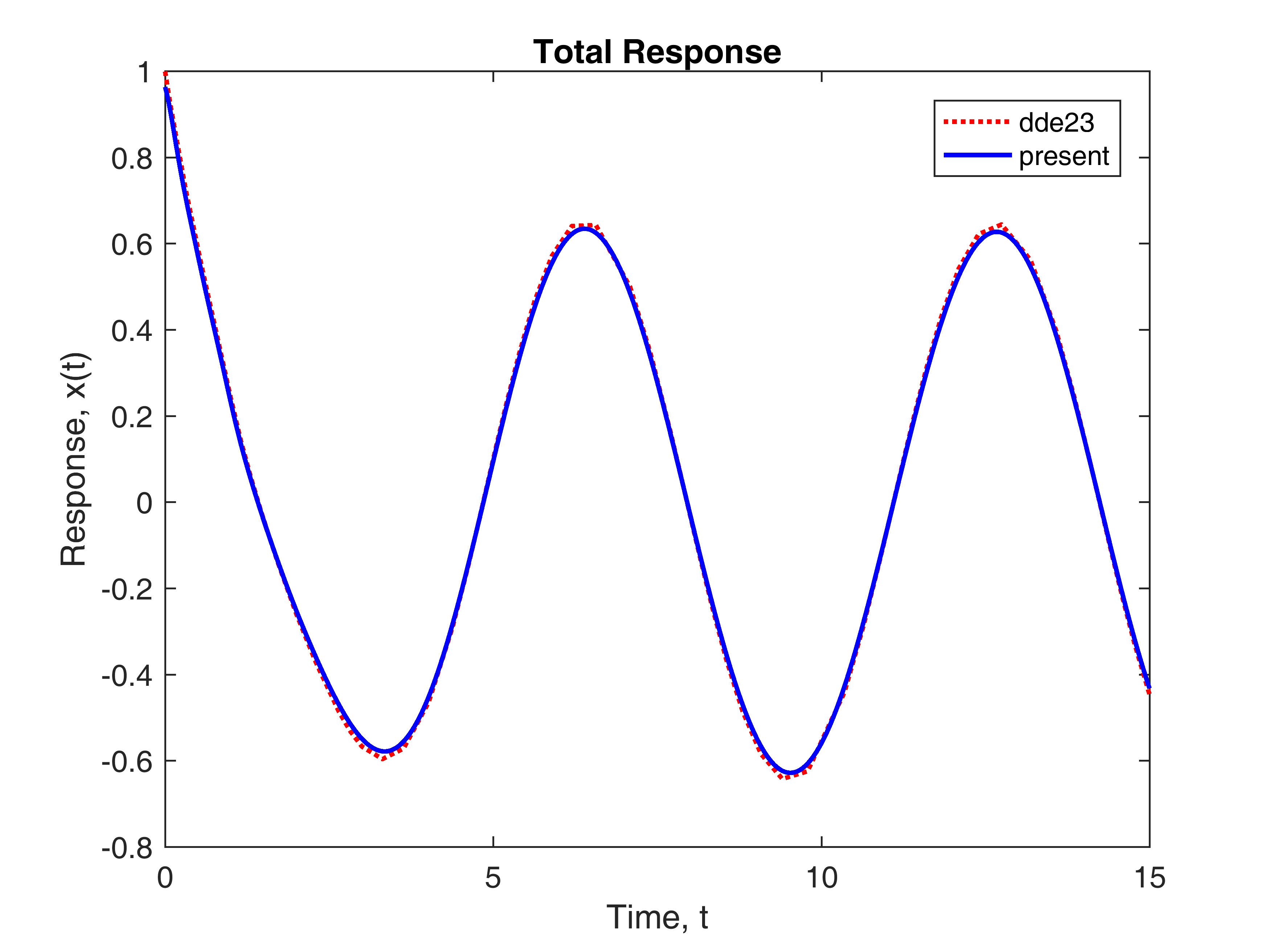}\\
	 (a) Initial response & (b) Total response\\
	\end{tabular}
\caption{Time responses of the single delay system with $a=-1$ and $a_{1d}=-1$.}
\label{Fig_LambertW_1}
\end{figure}

\subsection{Two-delay case}

Consider the following differential equation with two delays:
\begin{align*}
\dot{x}(t)=-x(t)- x(t-1) - \frac12 x(t-2)+ \cos t, t>0 \\
x(0)=x_0=1,\quad x(t)=\phi(t)=1, t\in[-2,0)
\end{align*}
The corresponding eigenvalues $S_n$ are computed in \textbf{Case (a)} of Example \ref{exmp1} in previous section. 
Since
\[
\Phi_1(s)=\int_{-1}^0 \phi(t) e^{-st} dt=\frac{e^s-1}{s},\quad
\Phi_2(s)=\int_{-2}^0 \phi(t) e^{-st} dt=\frac{e^{2s}-1}{s},
\]
thus coefficients for time responses are given by
\begin{align*}
C_n&=1/(1-e^{-S_n}-\frac12e^{-2S_n}) \\
C^I_n&=C_n \left(-e^{-S_n}\Phi_1(S_n)-\frac12 e^{-2 S_n}\Phi_2(S_n)\right)
\end{align*}
and
their values are described in Table \ref{tab:LambertW_2_mode}. 
The state transition function is then approximated by
\[
\Psi(t)=\sum_{n=-\infty}^{\infty} C^N_n e^{-S_n t}
\approx\sum_{n=-11}^{n=10} C^N_n e^{-S_n t}
\]

The first three branches (i.e., $k=0, 1, 2$ ) contribution to the initial condition response are shown in Fig. \ref{Fig_LambertW_2_mode}(a).
It is obviously, the mode shape due to the branch $k=0$ is similar to the solution
obtained by \texttt{dde23} of Matlab command which is referred here as exact solution.
And if we increase the number of branches in approximate the time response, then
the maximal absolute error between computed solution and the one obtained by \texttt{dde23} is shown in Fig. \ref{Fig_LambertW_2_mode}(b). This figure shows as the number of branches increases the error is reduced as well with slow rate. 
Therefore the total response of this two-delay system is then given by
\[
x(t)\approx \sum_{n=-11}^{10} (C^N_n +C^{IN}_n) e^{-S_n t}
+\sum_{n=-11}^{10} C^N_n \int_0^t e^{-S_n (t-\tau)} \cos(\tau) d\tau
\]
and the first term gives the initial condition response. The initial condition response and total response of this system are given in Fig. \ref{Fig_LambertW_2} which
are very close to the responses produced by \texttt{dde23}.

\begin{figure}[h]
	\hspace*{-2em}
	\begin{tabular}{cc}
		\includegraphics[width=7.25cm]{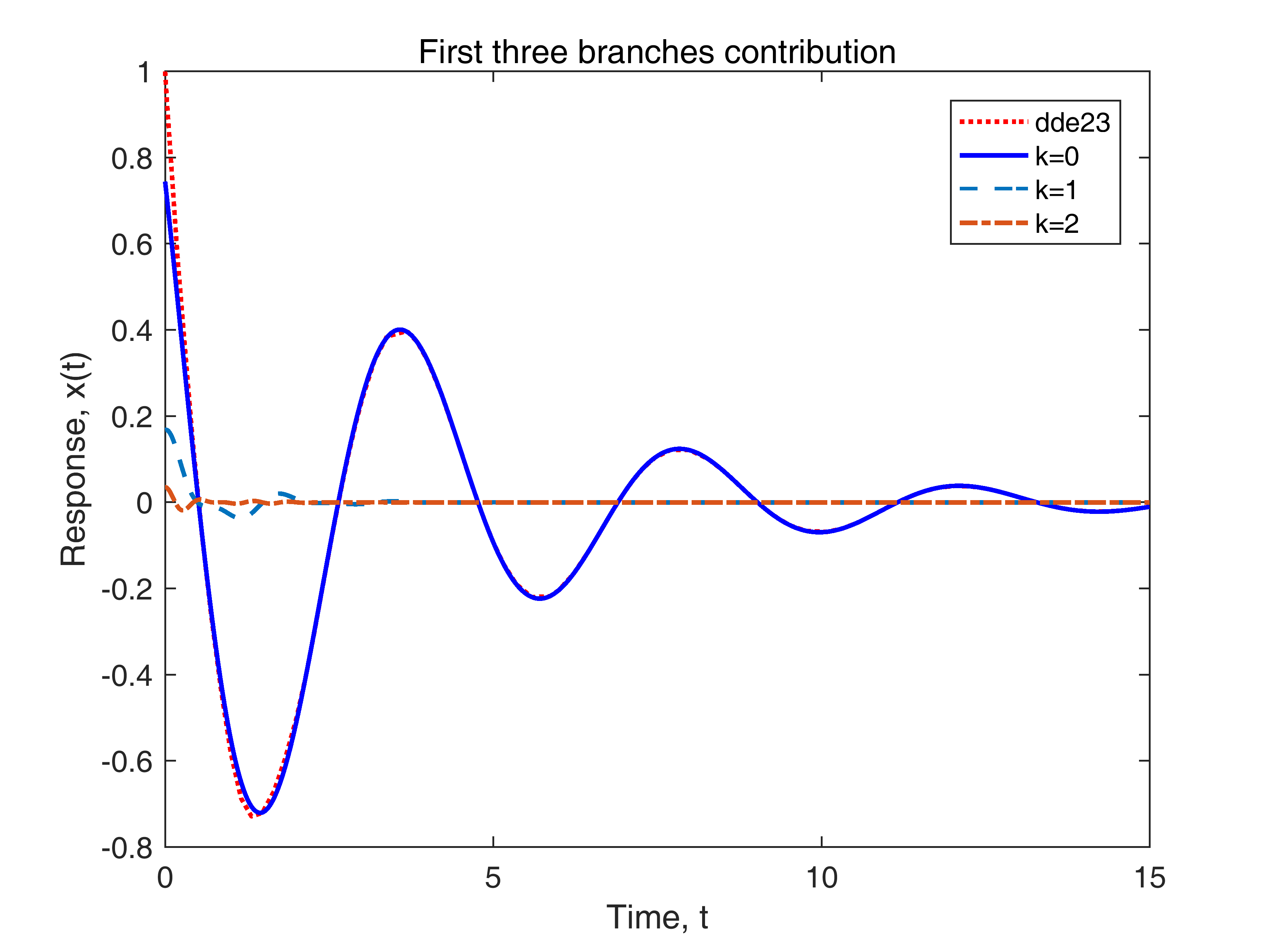} &
	\hspace*{-3em}
		\includegraphics[width=7.25cm]{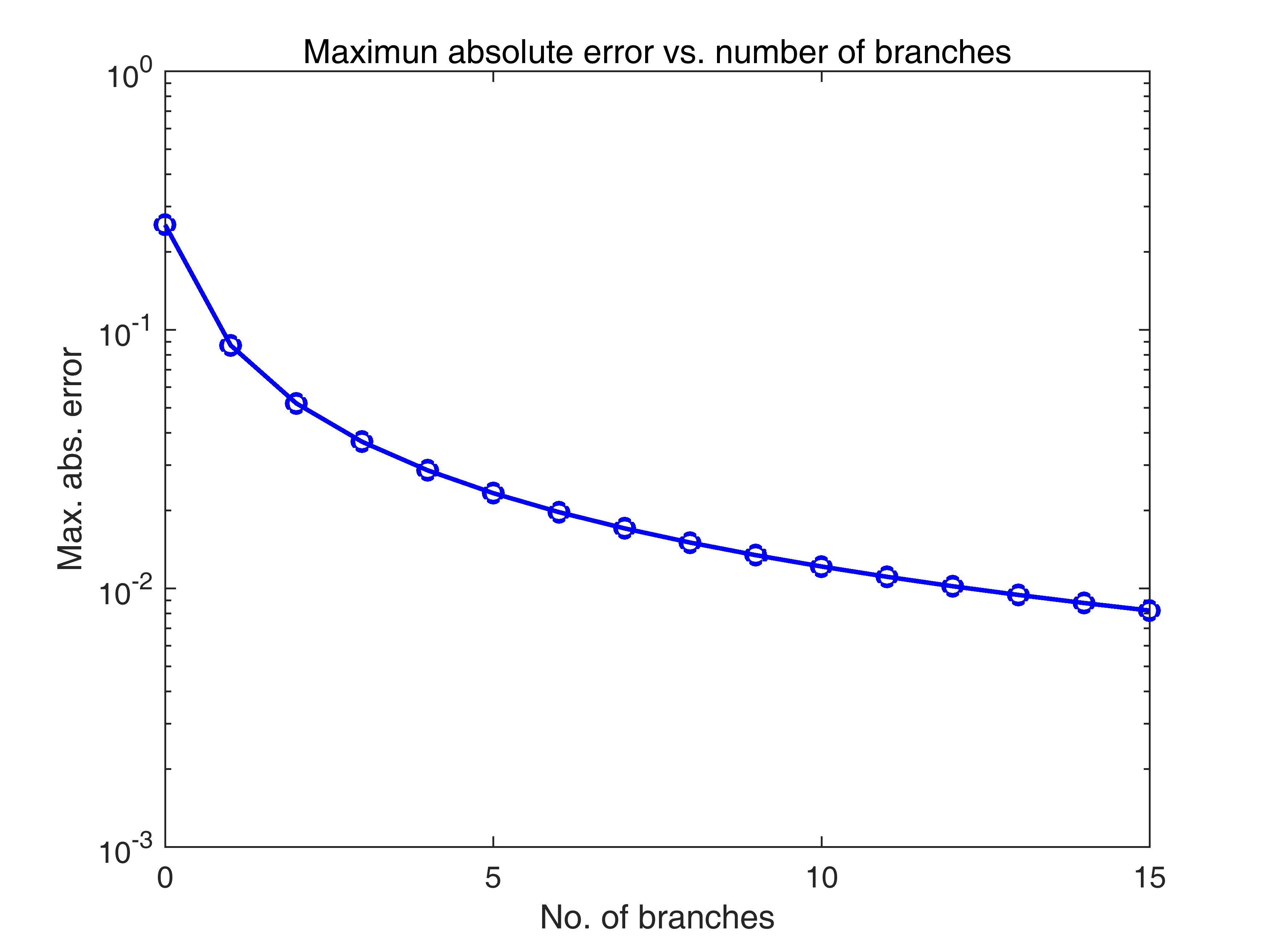}\\
		(a) Effect of first three branches ($k=0, 1, 2$) & (b) Error vs. no. of branches
	\end{tabular}
	\caption{Modal effect of initial responses due to branches of Lambert W function}
	\label{Fig_LambertW_2_mode}
\end{figure}

\begin{figure}[h]
	\hspace*{-2em}
	\begin{tabular}{cc}
		\includegraphics[width=7.25cm]{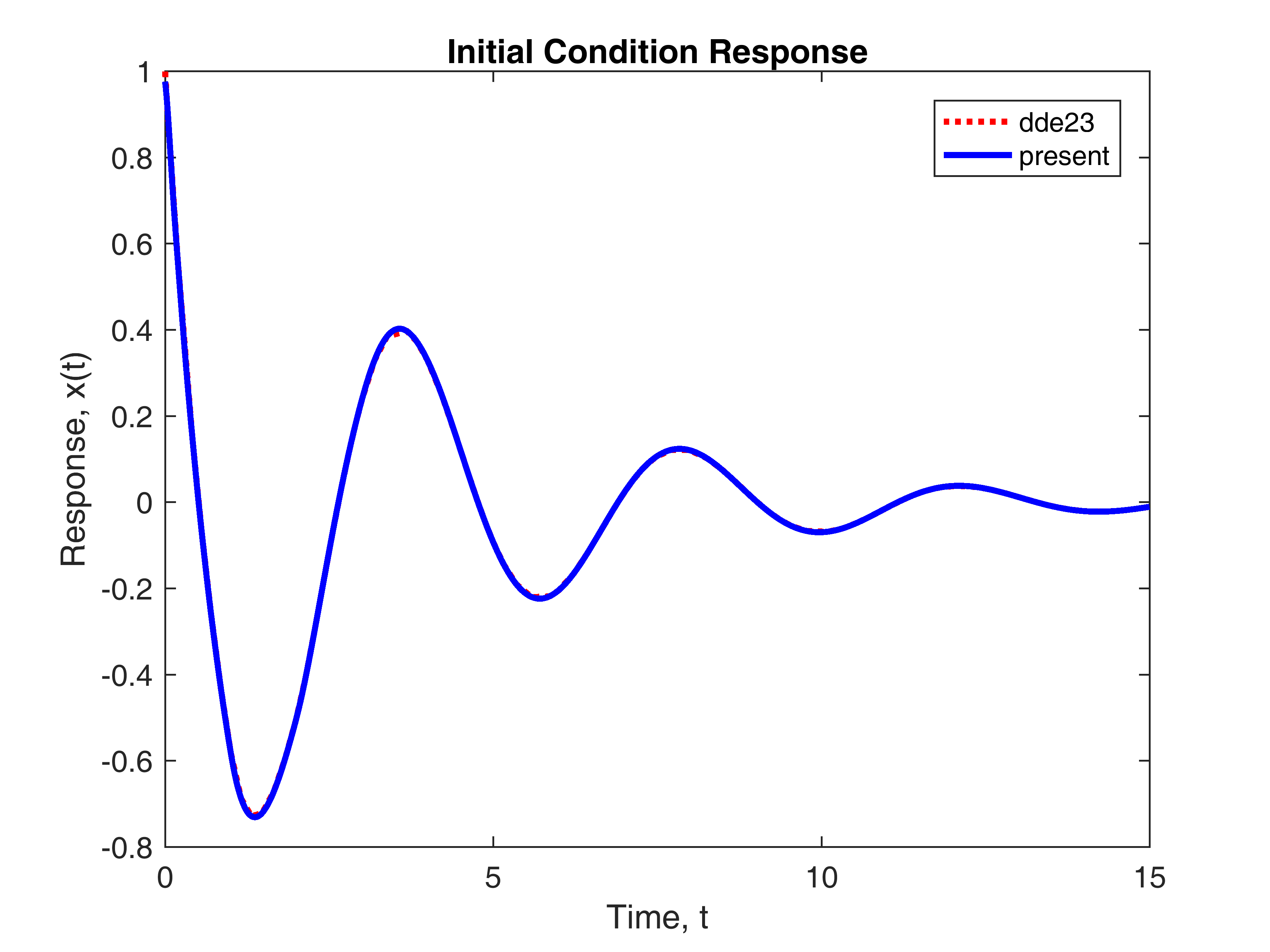} &
  	\hspace*{-3em}
		\includegraphics[width=7.25cm]{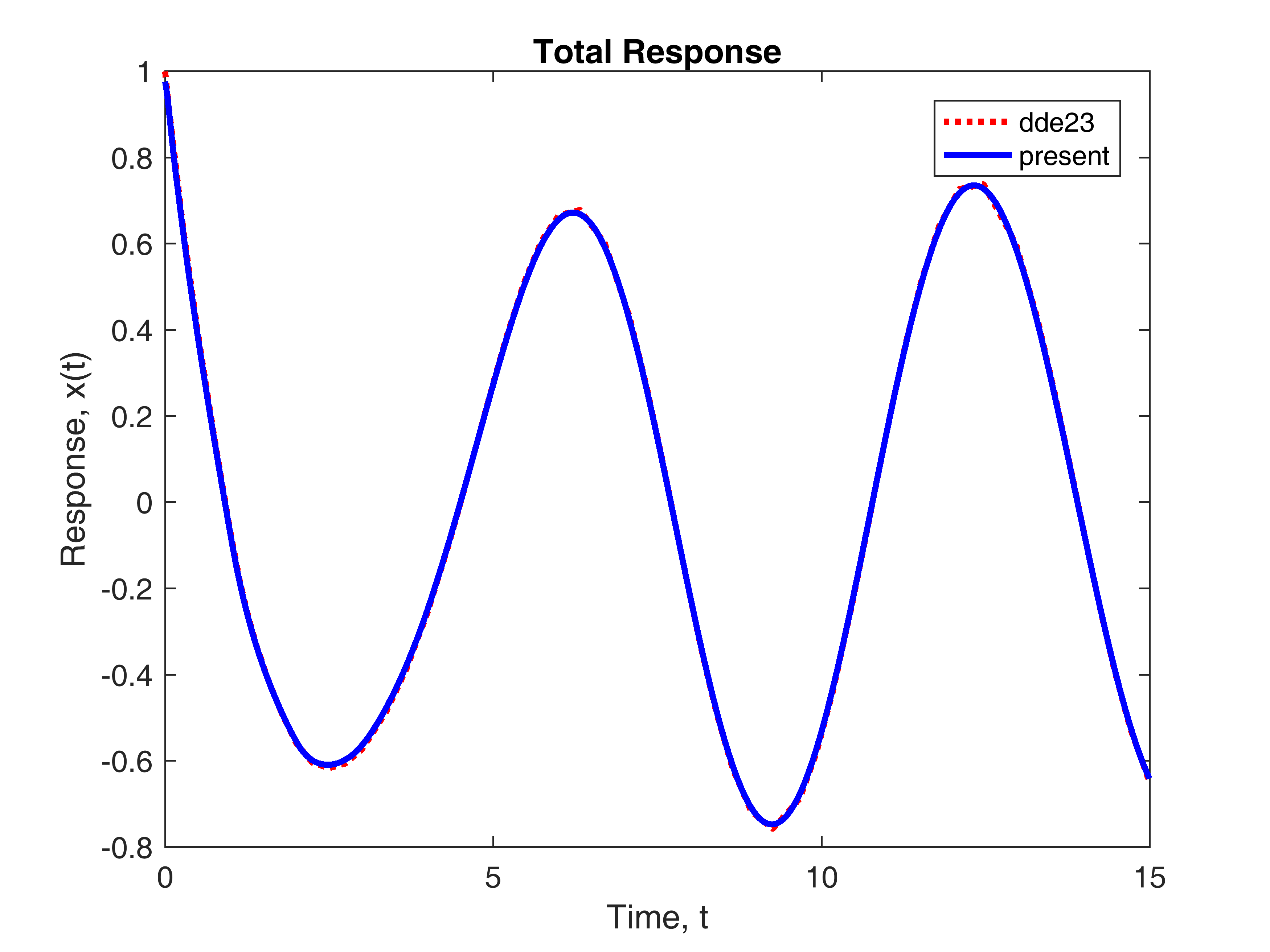}\\
		(a) Initial response & (b) Total response
	\end{tabular}
	\caption{Time Responses of the two-delay systems with $a=-1$, $a_{1d}=-1$, and $a_{2d}=-1/2$.}
	\label{Fig_LambertW_2}
\end{figure}

\begin{landscape}
\renewcommand{\arraystretch}{2} 
\setlength\arraycolsep{1.5pt}
\begin{table}
	\caption{Eigenvalues $S_n$ with associated coefficients $C_n$ and $C_n^I$ for response computation }
	\label{tab:LambertW_2_mode}
	\begin{tabular}{|c|c|c|}
		\hline
		$S_n$ & $C_{n}$ & $C_{n}^I$ \\
		\hline
	$S_{0}=s_{1,0},S_{-1}=s_{2,0} \approx -0.27495\pm1.47517i$ &
    $C_{0},  C_{-1}   \approx  0.27626  \mp 0.17588 i$ &
    $C_{0}^I,C_{-1}^I \approx  0.096135 \pm 0.57465 i$   \\
	$S_{1}=s_{1,1},S_{-2}=s_{1,-1}\approx -1.14682\pm7.24009i$, &
    $C_{1},  C_{-2}   \approx  0.017012 \mp 0.080385 i$ &
    $C_{1}^I,C_{-2}^I \approx  0.010973 \pm 0.081826 i$	\\	
	$S_{2}=s_{2,1},S_{-3}=s_{2,-1}\approx -1.27049\pm3.64513i$ &
    $C_{2},  C_{-3}   \approx -0.030157 \mp 0.10209 i$ & 
    $C_{2}^I,C_{-3}^I \approx  0.086161 \pm  0.061884 i$ \\	
	$S_{3}=s_{1,2},S_{-4}=s_{1,-2}\approx -1.50747\pm13.4657i$ &
    $C_{3},  C_{-4}   \approx  0.0050331 \mp 0.042141 i$ &
    $C_{3}^I,C_{-4}^I \approx  0.0027971 \pm 0.042198 i$	\\
	$S_{4}=s_{2,2},S_{-5}=s_{2,-2}\approx -1.66428\pm10.0261i$ &
    $C_{4},  C_{-5}   \approx -0.0084729 \mp -0.041680 i$ &
    $C_{4}^I,C_{-5}^I \approx  0.018246  \pm  0.037945 i$ \\
	$S_{5}=s_{1,3},S_{-6}=s_{1,-3}\approx -1.77287\pm19.72659i$ &
    $C_{5},  C_{-6}   \approx  0.0025181 \mp 0.028227 i$ &
    $C_{5}^I,C_{-6}^I \approx  0.0010598 \pm 0.028234 i$ \\
	$S_{6}=s_{2,3},S_{-7}=s_{2,-3}\approx -1.87170\pm16.34717i$ &
    $C_{6},  C_{-7}   \approx -0.0043406 \mp 0.026636 i$ &
    $C_{6}^I,C_{-7}^I \approx  0.0082865 \pm 0.025521 i$ \\
	$S_{7}=s_{1,4},S_{-8}=s_{1,-4}\approx -1.87589\pm25.99734i$ &
    $C_{7},  C_{-8}   \approx  0.0015665  \mp 0.021147 i$ &
    $C_{7}^I,C_{-8}^I \approx  0.00046741 \pm 0.021151 i$ \\
	$S_{8}=s_{2,4},S_{-9}=s_{2,-4}\approx -2.01488\pm22.65195i$ &
    $C_{8},  C_{-9}   \approx -0.0027376 \mp 0.0196560194) i$ &
    $C_{8}^I,C_{-9}^I \approx  0.0048633 \pm 0.019165 i$ \\
	$S_{9}=s_{1,5},S_{-10}=s_{1,-5}\approx -1.99436\pm32.27231i$ &
    $C_{9},  C_{-10}   \approx  0.0010945 \mp 0.016880 i$&
    $C_{9}^I,C_{-10}^I \approx  0.00021337 \pm 0.016884 i$ \\
	$S_{10}=s_{2,5},S_{-11}=s_{2,-5}\approx -2.124656\pm28.94953i$ &
    $C_{10},  C_{-11}   \approx -0.0019236 \mp 0.015603 i$ &
    $C_{10}^I,C_{-11}^I \approx  0.0032516 \pm 0.015339 i$ \\
    \hline		
	\end{tabular}
\end{table}
\end{landscape}

\renewcommand{\arraystretch}{1} 
\setlength\arraycolsep{0.2em}

\subsection{Three-delay case}

Consider the following differential equation with three delays:
\begin{align*}
\dot{x}(t)=-x(t)+\frac12 x(t-1) - x(t-2) - x(t-3)+ \cos t, t>0 \\
x(0)=x_0=1,\quad x(t)=\phi(t)=1, t\in[-3,0)
\end{align*}
The corresponding central 14 eigenvalues $S_n$ is computed in Example \ref{exmp2} of previous section. 
Since
\[
\Phi_j(s)=\int_{-j}^0 \phi(t) e^{-st} dt=\frac{e^{j s}-1}{s}, \quad 1\le j\le 3,
\]
thus coefficients for time responses are given by
\begin{align*}
C_n&=1/(1+\frac12 e^{-S_n}- e^{-2S_n}- e^{-2S_n}), \\
C^I_n&=C_n \left(\frac12 e^{-S_n}\Phi_1(S_n)- e^{-2 S_n}\Phi_2(S_n)- e^{-3 S_n}\Phi_3(S_n)\right)
\end{align*}
and their values are omitted here. We use eigenvalues corresponding to the branches $k=0,\pm1,\ldots,\pm 5$ of Lambert W function for approximation (i.e., we use
32 eigenvalues in total), and thus
the state transition function is approximated by
\[
\Psi(t)=\sum_{n=-\infty}^{\infty} C^N_n e^{-S_n t}
\approx\sum_{n=-11}^{n=10} C^N_n e^{-S_n t}
\]
and the total response of this two-delay system is then given by
\[
x(t)\approx \sum_{n=-15}^{15} (C^N_n +C^{IN}_n) e^{-S_n t}
+\sum_{n=-15}^{15} C^N_n \int_0^t e^{-S_n (t-\tau)} \cos(\tau) d\tau
\]
with the first term for the initial condition response. The initial condition response and total response of this system are given in Fig. \ref{Fig_LambertW_3} which
are almost distinguishable to the responses generated by \texttt{dde23}.

\begin{figure}[h]
	\hspace*{-2em}
	\begin{tabular}{cc}
		\includegraphics[width=7.25cm]{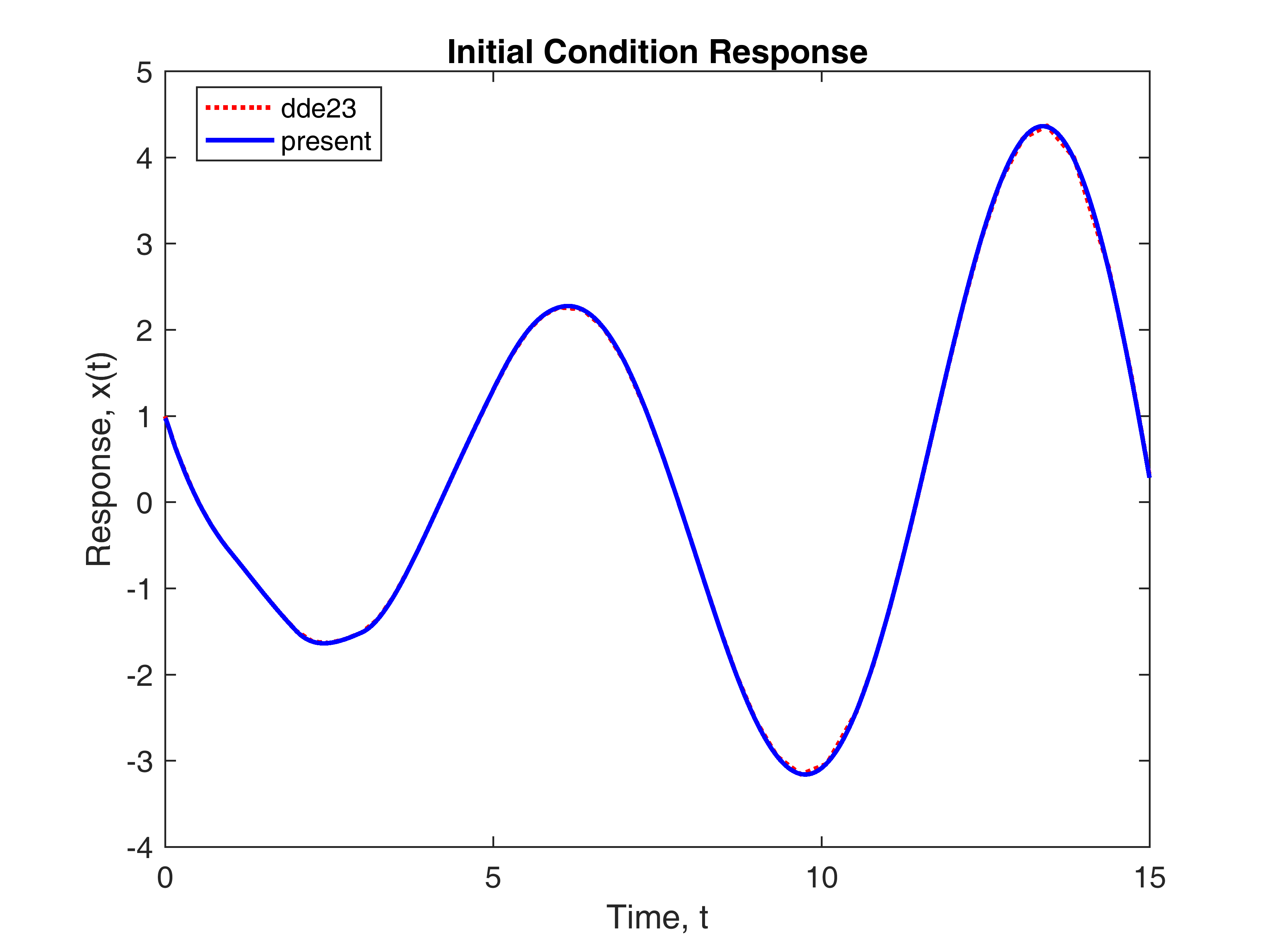} &
		\hspace*{-3em}
		\includegraphics[width=7.25cm]{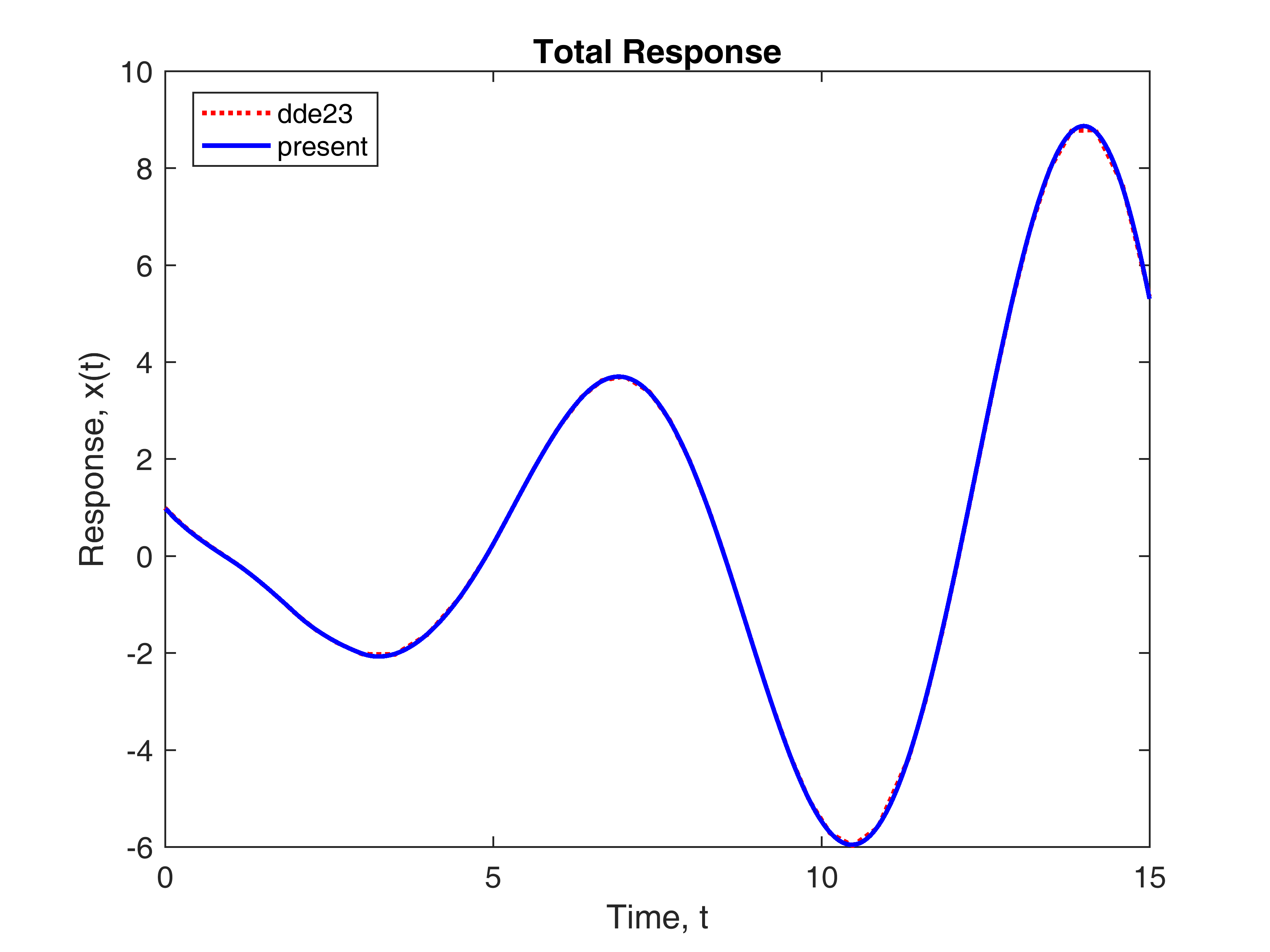}\\
		(a) Initial response & (b) Total response
	\end{tabular}
	\caption{Time Responses of the three-delay systems with $a=-1$, $a_{1d}=1/2$, $a_{2d}=-1$, and $a_{3d}=-1$.}
	\label{Fig_LambertW_3}
\end{figure}

\section{Conclusion}
Lambert W function has been widely applied in the research of science and engineering problems. In 
this paper, we construct the time response of scalar systems with multiple 
discrete delays based on the Laplace transform. We first establish the state
transition function as the sum of infinite series of exponentials acting on 
eigenvalues inside countable branches of the Lambert W functions. And 
the time response function is then obtained in terms of initial condition, 
preshape function, and the convolution between state transition function and 
input function. Eigenvalues in each branch of Lambert W function are computed 
by a numerical iteration. Once the eigenvalues are obtained, the 
approximation of time responses is calculated. Numerical examples are 
presented to conform the successful application of Lambert W function in 
establishing the time response of multi-delay systems. Our main contribution
is to construct an analytical formula in representing the time response which
can act as the concise foundation of designing the feedback controller for 
multi-delay systems.

\subsection*{Acknowledgments}
This research was partially supported by grant MOST 104-2115-M-029-004.
 (H.-N. Huang).


\begin{thebibliography}{99}

\bibitem{Kolmanovskii1} V. Kolmanovskii, A. Myshkis;
\emph{Introduction to the Theory and Applications of Functional Differential 
Equations}, Kluwer Academic Publishers, Dordrecht, 1999.

\bibitem{Kolmanovskii2} V. Kolmanovskii, A. Myshkis;
\emph{Applied Theory of Functional Differential Equations}, 
Kluwer Academic Press Publisher, Dordrect, 1992.

\bibitem{kuang1} S. A. Gourley, Y. Kuang;
 A stage structured predator-prey model and its dependence on matu-ration 
delay and death rate, \emph{J. Math. Biol.} \textbf{49}(2004), 188--200.

\bibitem{kuang2} Y. Kuang;
\emph{Delay Differential Equations with Applications in Population Dynamics}, 
Academic Press, New York, 1993.

\bibitem{murray} J. D. Murray;
\emph{Mathematical Biology: An Introduction}, Springer-Verlag, 
3rd Ed., New York, 2002

\bibitem{rws} R. W.  Shonkkwiler, J. Herod;
\emph{Mathematical Biology, an Introduction whith Mapple and Matlab}, 
2nd Ed., Springer, San Francisco, 2009.

\bibitem{smith} H. Smith;
 \emph{An Introduction to Delay Differential Equations with Applications 
to the Life Sciences}, Springer, New York, 2011.

\bibitem{travis} S. P. Travis;
 A one-dimensional two-body problem of classical electrodinamics, 
\emph{J. Appl. Math.} \textbf{28}(1975), 611--632.

\bibitem{Coreless}
R. M. Coreless, G. H. Gonnet, D. E. G. Hare, D. J. Jerey, D. E. Knuth,
On the Lambert W function, \emph{Adv. Comput. Math.} \textbf{5}(1996), 329--359. 

\bibitem{Asl}
F. M. Asl, A. G. Ulsoy, Analysis of a system of linear delay differential equations, \emph{Journal of Dynamic Systems, Measurement, and Control} \textbf{125}(2003), 215--223.

\bibitem{SunYi} S. Yi, P. W. Nelson, A. G. Ulsoy,
\emph{Time-Delay Systems: Analysis and Control Using the Lambert W Function},
World Scientific, Singapore, 2010.

\bibitem{KeqinGu} K. Gu, V. L. Kharitonov, J. Chen,
\emph{Stability of Time-Delay Systems}, Birkh\:{a}user, Boston, 2003.

\bibitem{Krisztin} T. Krisztin, G. Vas, On the fundamental solution
of linear delay differential equations with multiple delays,
\emph{Electronic Journal of Qualitative Theory of Differential Equations}
\textbf{2011}, No.36, 1-28.

\bibitem{Shinozaki}
H. Shinozaki, T. Mori, Robust stability analysis of linear time-delay systems
by Lambert W function: Some extreme point results,
\emph{Automatica}, 43 (2006) 1791-1799.

\bibitem{Bellman}
R. Bellman, K. L. Cooke, \emph{Differential Difference Equations},
 Academic Press, New York, 1963.

\end{thebibliography}
\end{document}